\documentclass[reqno,11pt]{amsart}
\usepackage{amscd,amssymb,verbatim,array}
\usepackage{hyperref}

\numberwithin{equation}{section} \theoremstyle{plain}

\newcommand{\Complex}{\mathbb C}
\newcommand{\Real}{\mathbb R}

\newcommand{\ddbar}{\overline\partial}
\newcommand{\pr}{\partial}
\newcommand{\ol}{\overline}
\newcommand{\Td}{\widetilde}
\newcommand{\norm}[1]{\left\Vert#1\right\Vert}
\newcommand{\set}[1]{\left\{#1\right\}}
\newcommand{\To}{\rightarrow}
\DeclareMathOperator{\Ker}{Ker}

\newtheorem{theorem}{Theorem}[section]
\newtheorem{lemma}[theorem]{Lemma}

\newtheorem{corollary}[theorem]{Corollary}
\newtheorem{definition}[theorem]{Definition}

\theoremstyle{definition}

\theoremstyle{remark}

\numberwithin{equation}{section}

\newcommand{\abs}[1]{\lvert#1\rvert}

\begin{document}

\title[The asymptotics of the analytic torsion on CR manifolds with $S^1$ action]
{The asymptotics of the analytic torsion on CR manifolds with $S^1$ action}

\author{Chin-Yu Hsiao}

\address{Institute of Mathematics, Academia Sinica and National Center for Theoretical Sciences, Astronomy-Mathematics Building, No. 1, Sec. 4, Roosevelt Road, Taipei 10617, Taiwan}
\thanks{The first author was partially supported by Taiwan Ministry of Science of Technology project 104-2628-M-001-003-MY2 and the Golden-Jade fellowship of Kenda Foundation. The second author was supported by Taiwan Ministry of Science of Technology project 103-2115-M-008-008-MY2}

\email{chsiao@math.sinica.edu.tw or chinyu.hsiao@gmail.com}

\author{Rung-Tzung Huang}

\address{Department of Mathematics, National Central University, Chung-Li 320, Taiwan}

\email{rthuang@math.ncu.edu.tw}


\begin{abstract}

Let $X$ be a compact connected strongly pseudoconvex CR manifold of dimension $2n+1, n \ge 1$ with a transversal CR $S^1$-action on $X$. We introduce the Fourier components of the Ray-Singer analytic torsion on $X$ with respect to the $S^1$-action. We establish an asymptotic formula for the Fourier components of the analytic torsion with respect to the $S^1$-action. This generalizes the asymptotic formula of Bismut and Vasserot on the holomorphic Ray-Singer torsion associated with high powers of a positive line bundle to strongly pseudoconvex CR manifolds with a transversal CR $S^1$-action. 

\end{abstract}


\maketitle


\section{Introduction}

In \cite{RS2}, Ray and Singer introduced the holomorphic analytic torsion for $\overline{\partial}$-complex on complex manifolds as the complex analogue of the analytic torsion for flat vector bundles \cite{RS1}. In \cite{BV}, Bismut and Vasserot established the asymptotic formula of the holomorphic analytic torsion associated with high powers of a positive line bundle, by using the heat kernel method of \cite{B} (see also \cite[Sect. 5.5]{MM}). The asymptotic formula has applications in Arakelov geometry \cite{S}. Another applications of the holomorphic torsion appeared in the study of the moduli space of K3 surfaces by Yoshikawa in \cite{Y} and his subsequent works. Recently, M. Puchol \cite{P} extended the results of Bismut and Vasserot on the asymptotic of the holomorphic torsion to the fibration case.

In orbifold geometry, we have Kawasaki's Hirzebruch-Riemann-Roch formula~\cite{Ka79} and also general index theorem~\cite{PV}. To study further geometric problem for orbifolds (for example, local family index theorem), it is important to know the corresponding heat kernel asymptotics and to have the concept of analytic torsion. The difficulty comes from the fact that in general the heat kernel expansion involves strata contribution (see~\cite{DGGW08}) and it is difficult to study local index theorem and analytic torsion on general orbifolds. Thus, it is natural to first attack such problems on some class of orbifolds. In complex orbifold geometry, we usually consider an orbifold ample line bundle $L$ over a a projective orbifold $M$. An important case is when ${\rm Tot\,}(L^*)$ the
space of all non-zero vectors in the dual bundle of $L$ is smooth, where ${\rm Tot\,}(L^*)$ is considered as a complex manifold. In this case, the circle bundle $C(L^*)$ is a smooth manifold. Actually, $C(L^*)$ is a Sasakian manifold and all quasi-regular Sasakian manifolds are obtained this way (see~\cite{OV07}). Moreover, $C(L^*)$ is a smooth CR manifold with a transversal CR $S^1$ action. The point is that since $C(L^*)$ is a smooth manifold, if we work directly on $C(L^*)$, we should get better results than general orbifolds case (see ~\cite{Hsiao14},~\cite{HL1},~\cite{HsiaoLi15} and~\cite{CHT}). For example, in~\cite{CHT}, we can reinterpret Kawasaki's Hirzebruch-Riemann-Roch formula as an index theorem obtained by 
an integral over a smooth CR manifold which is essentially the circle bundle of this line bundle. Thus, we could first try to define and study analytic torsion on such CR manifolds. This is the motivation of this paper. 

To motivate our approach, let's come back to complex geometry case. Consider a compact smooth complex manifold $M$ of dimension $n$ and let $(L,h^L)\To M$ be a holomorpic line bundle over $M$, where $h^L$ denotes a Hermitian fiber metric of $L$. Let $(L^*,h^{L^*})\To M$ be the dual bundle of $(L,h^L)$ and put $X=\set{v\in L^*;\, \abs{v}^2_{h^{L^*}}=1}$. We call $X$ the circle bundle of $(L^*,h^{L^*})$. It is clear that $X$ is a compact CR manifold of dimension $2n+1$. Given a local holomorphic frame $s$ of $L$ on an open subset $U\subset M$ we define the associated local weight of $h^L$ by
$\abs{s(z)}^2_{h^L}=e^{-2\phi(z)}$, $\phi\in C^\infty(U, \Real)$. The CR manifold $X$ is equipped with a natural  $S^1$ action. Locally $X$ can be represented in local holomorphic coordinates $(z,\lambda)\in\mathbb C^{n+1}$, where $\lambda$ is the fiber coordinate, as the set of all $(z,\lambda)$ such that $\abs{\lambda}^2e^{2\phi(z)}=1$,
where $\phi$ is a local weight of $h^L$. The $S^1$ action on $X$ is given by $e^{i\theta}\circ (z,\lambda)=(z,e^{i\theta}\lambda)$, $e^{i\theta}\in S^1$, $(z,\lambda)\in X$.  Let $T\in C^\infty(X,TX)$ be the real vector field induced by the $S^1$ action, that is, $Tu=\frac{\pr}{\pr\theta}(u(e^{i\theta}\circ x))|_{\theta=0}$, $u\in C^\infty(X)$. We can check that $[T,C^\infty(X,T^{1,0}X)]\subset C^\infty(X,T^{1,0}X)$ and $\Complex T(x)\oplus T^{1,0}_xX\oplus T^{0,1}_xX=\Complex T_xX$ (we say that the $S^1$ action is CR and transversal). For every $m\in\mathbb Z$, put 
\[\begin{split}
\Omega^{0,\bullet}_m(X):&=\set{u\in\Omega^{0,\bullet}(X);\, Tu=imu}\\
&=\set{u\in\Omega^{0,\bullet}(X);\, u(e^{i\theta}\circ x)=e^{im\theta}u(x), \forall\theta\in[0,2\pi[}.\end{split}\]
Since $\ddbar_bT=T\ddbar_b$, we have $\ddbar_b:\Omega^{0,\bullet}_m(X)\To\Omega^{0,\bullet}_m(X)$, where $\ddbar_b$ denotes the tangential Cauchy-Riemann operator.
Let $\Omega^{0,\bullet}(M,L^m)$ be the space of smooth sections of $(0,\bullet)$ forms of $M$ with values in $L^m$, where $L^m$ is the $m$-th power of $L$. It is well-known that (see Theorem 1.2 in~\cite{CHT}) there is a bijection map 
\[A_m:\Omega^{0,\bullet}_m(X)\To\Omega^{0,\bullet}(M,L^m)\]
 such that $A_m\ddbar_{b}=\ddbar A_m$ on $\Omega^{0,\bullet}_m(X)$. Let $\Box_m$ be the Kodaira Laplacian with values in $T^{*0,\bullet}M\otimes L^m$ and let $e^{-t\Box_m}$ be the associated heat operator. It is well-known that $e^{-t\Box_m}$ admits an asymptotic expansion as $t\To0^+$. Consider $B_m(t):=(A_m)^{-1}\circ e^{-t\Box_m}\circ A_m$. Let 
 \[\Box_{b,m}:\Omega^{0,\bullet}_m(X)\To\Omega^{0,\bullet}_m(X)\] 
 be the Kohn Laplacian for forms with values in the $m$-th $S^1$ Fourier component and let $e^{-t\Box_{b,m}}$ be the associated heat operator.  We can check that  
\begin{equation}\label{e-gue150923bi}
 e^{-t\Box_{b,m}}=B_m(t)\circ Q_m=Q_m\circ B_m(t)\circ Q_m,
\end{equation}
where $Q_m:\Omega^{0,\bullet}(X)\To\Omega^{0,\bullet}_m(X)$ is the orthogonal projection. From the asymptotic expansion of $e^{-t\Box_m}$ and \eqref{e-gue150923bi}, it is straightforward to see that
\begin{equation}\label{e-gue151108}
e^{-t\Box_{b,m}}(x,x)\sim t^{-n}a_n(x)+t^{-n+1}a_{n-1}(x)+\cdots.
\end{equation} 
From \eqref{e-gue151108}, we can define $\exp ( -\frac{1}{2} \theta_{b,m}'(0) )$ the $m$-th Fourier component of the analytic torsion on the CR manifold $X$, where $\theta_{b,m}(z)  = - M \left\lbrack \operatorname{STr}  \lbrack N e^{-t \Box_{b,m}} \Pi^\perp_m  \rbrack \right\rbrack$. Here $N$ is the number operator on $T^{*0,\bullet}X$, $\operatorname{STr}$ denotes the super trace , $\Pi^\perp_m$ is the orthogonal projection onto $({\rm Ker\,}\Box_{b,m})^\perp$ and $M$ denotes the Mellin transformation. Let $\exp ( -\frac{1}{2} \theta_{m}'(0) )$ be the analytic torsion associated to $L^m$. It is easy to see that 
\begin{equation}\label{e-gue160501}
\theta_{b,m}'(0)=\theta_{m}'(0).
\end{equation}
Assume that $R^L>0$, where $R^L$ is the curvature induced by $h^L$. Bismut and Vasserot's asymptotic formula (see \cite[Theorem 8]{BV} and \cite[Theorem 5.5.8]{MM}) tells us that  
as $m \to+\infty$, we have 
\begin{equation}\label{e-gue160502}
\theta_{m}'(0)  =\frac{1}{2}\int_M \log \left( \det \left(  \frac{m \dot{R}^L }{2\pi} \right) \right) e^{m\frac{\sqrt{-1}}{2\pi}R^L}  + o(m^{n}),
\end{equation}
where $\dot{R}^L \in \operatorname{End}(T^{1,0}M)$ is  the Hermitian matrix given by
\[R^L(V, \overline{W}) = \langle\,\dot{R}^L V\,|\,\overline{W}\,\rangle,\ \ V, W\in T^{1,0}M.\]
Here $\langle\,\cdot\,|\,\cdot\,\rangle$ is the fixed Hermitian metric on $\Complex TM$. Let's reformulate \eqref{e-gue160502} in terms of geometric objects on $X$: 
\begin{equation}\label{e-gue160502I}
\theta_{b,m}'(0) =  \frac{1}{4\pi}
 \int_X \log \left( \det \left(  \frac{m \dot{\mathcal{R}} }{2\pi} \right) \right) e^{-m\frac{d\omega_0}{2\pi}} \wedge (-\omega_0) + o(m^{n}),
 \end{equation}
 where $\omega_0$ is the unique one form given by \eqref{e-gue160502b}, and $\dot{\mathcal{R}}\in \operatorname{End}(T^{1,0}X)$ is  the Hermitian matrix given by
\[id\omega_0(V, \overline{W}) = \langle\,\dot{\mathcal{R}} V\,|\,\overline{W}\,\rangle,\ \ V, W\in T^{1,0}X.\]
The purpose of this paper is to define the Ray-Singer torsion and establish \eqref{e-gue160502I} on any abstract strongly pseudoconvex CR manifols with a transversal CR locally free $S^1$-action. Note that for the case of circle bundle, the $S^1$ action is globally free and $X$ is strongly pseudoconvex if $R^L>0$. 

\subsection{The difficulty}\label{s-gue160502}

The proof of Bismut and Vasserot's asymptotic formula is based on the following results: 
\begin{equation}\label{e-gue160502a}
e^{-\frac{t}{m}\Box_m}(z,z)=(2\pi)^{-n} \frac{\det(\dot{R}^L)\exp(t\omega_d)}{\det(1-\exp(-t\dot{R}^L))}(z)m^n+o(m^n), 
\end{equation}
where $\omega_d = - \sum_{l,j}R^L(w_j, \overline{w}_l) \overline{w}^l \wedge \iota_{\overline{w}_j}$ and $\{w_j \}_{j=1}^n$ is a local orthonormal frame of $T^{1,0}M$ with dual frame $\{ w^j \}_{j=1}^n$, and 
\begin{equation}\label{e-gue160502aI}
m^{-n}\int_M{\rm STr\,}Ne^{-\frac{t}{m}\Box_m}(z,z)dv_M(z)\sim\sum ^\infty_{j=-n}B_{m,j}t^j\ \ \mbox{as $t\To0^+$, uniformly in $m$}, 
\end{equation}
where we can compute the term $\lim_{m\To\infty}B_{m,j}$, for each $j$. For the case where the $S^1$ action is only locally free, our asymptotic expansion involves an unprecedented contribution in terms of a distance function from lower dimensional strata of the $S^1$ action (see Theorem~\ref{T:1.6.1}).
Roughly speaking, on the regular part of $X$, we have 
\begin{equation}\label{e-gue160502aII}
e^{-t\Box_{b,m}}(x,x)=(2\pi)^{-n-1} \frac{\det(\dot{\mathcal{R}}) \exp(t \gamma_d)}{\det(1-\exp(-t\dot{\mathcal{R}}))}(x)+o(m^n)\mod O\Bigr(m^ne^{-md(x,X_{{\rm sing\,}})^2}\Bigr).
\end{equation} 
Moreover, as \eqref{e-gue160502aI}, integrating ${\rm STr\,}Ne^{-\frac{t}{m}\Box_{b,m}}(x,x)$ over $X$, we obtain an asymptotic expansion non-trivial the fractional power in $t^{\frac{1}{2}}$ (see Theorem~\ref{t-gue160427}).  
It turns out that this is a non-trivial generalization.


\subsection{Main result}\label{s-gue150508a}

We now formulate the main results. We refer to Section~\ref{s:prelim} for some notations and terminology used here. 

Let $(X, T^{1,0}X)$ be a compact connected strongly pseudoconvex CR manifold with a transversal CR locally free $S^1$ action $e^{i\theta}$ (see Definition~\ref{d-gue160502}), where $T^{1,0}X$ is a CR structure of $X$. 
Let $T\in C^\infty(X,TX)$ be the real vector field induced by the $S^1$ action and let $\omega_0\in C^\infty(X,T^*X)$ be the global real one form determined by 
\begin{equation}\label{e-gue160502b}
\langle\,\omega_0\,,\,T\,\rangle=-1,\ \ \langle\,\omega_0\,,\,u\,\rangle=0,\ \ \forall u\in T^{1,0}X\oplus T^{0,1}X.
\end{equation} 
For $x\in X$, we say that the period of $x$ is $\frac{2\pi}{\ell}$, $\ell\in\mathbb N$, if $e^{i\theta}\circ x\neq x$, for every $0<\theta<\frac{2\pi}{\ell}$ and $e^{i\frac{2\pi}{\ell}}\circ x=x$. For each $\ell\in\mathbb N$, put 
\begin{equation}\label{e-gue150802bm}
X_\ell=\set{x\in X;\, \mbox{the period of $x$ is $\frac{2\pi}{\ell}$}}
\end{equation} 
and let $p=\min\set{\ell\in\mathbb N;\, X_\ell\neq\emptyset}$. It is well-known that if $X$ is connected, then $X_p$ is an open and dense subset of $X$ (see Duistermaat-Heckman~\cite{Du82}). 
In this work, we assume that $p=1$ 
and we denote $X_{{\rm reg\,}}:=X_{p}=X_1$. We call $x\in X_{{\rm reg\,}}$ a regular point of the $S^1$ action. Let $X_{{\rm sing\,}}$ be the complement of $X_{{\rm reg\,}}$. 

Let $E$ be a rigid CR vector bundle over $X$ (see Definition~\ref{d-gue150508dI}) and we take a rigid Hermitian metric $\langle\,\cdot\,|\,\cdot\,\rangle_E$ on $E$ (see Definition~\ref{d-gue150514f}). Take a rigid Hermitian metric $\langle\,\cdot\,|\,\cdot\,\rangle$ on $\Complex TX$ such that $T^{1,0}X\perp T^{0,1}X$, $T\perp (T^{1,0}X\oplus T^{0,1}X)$, $\langle\,T\,|\,T\,\rangle=1$ and let $\langle\,\cdot\,|\,\cdot\,\rangle_E$ be the Hermitian metric on $T^{*0,\bullet}X\otimes E$ induced by the fixed Hermitian metrics on $E$ and $\Complex TX$. 
We denote by $dv_X=dv_X(x)$ the volume form on $X$ induced by the Hermitian metric $\langle\,\cdot\,|\,\cdot\,\rangle$ on $\Complex TX$. Then we get natural global $L^2$ inner product $(\,\cdot\,|\,\cdot\,)_{E}$ on $\Omega^{0,\bullet}(X,E)$. We denote by $L^{2}(X,T^{*0,\bullet}X\otimes E)$ the completion of $\Omega^{0,\bullet}(X,E)$ with respect to $(\,\cdot\,|\,\cdot\,)_{E}$. For every $u\in\Omega^{0,\bullet}(X,E)$, we can define $Tu\in\Omega^{0,\bullet}(X,E)$ and we have $T\ddbar_b=\ddbar_bT$.
For $m\in\mathbb Z$, put 
\[\begin{split}
\Omega^{0,\bullet}_m(X,E):&=\set{u\in\Omega^{0,\bullet}(X,E);\, Tu=imu}\\
&=\set{u\in\Omega^{0,\bullet}(X,E);\, (e^{i\theta})^*u=e^{im\theta}u,\ \ \forall\theta\in[0,2\pi[},\end{split}\]
where $(e^{i\theta})^*$ denotes the pull-back map by $e^{i\theta}$ (see \eqref{e-gue150508faI}). 
For each $m\in\mathbb Z$, we denote by $L^{2}_m(X,T^{*0,\bullet}X\otimes E)$ the completion of $\Omega^{0,\bullet}_m(X,E)$ with respect to $(\,\cdot\,|\,\cdot\,)_{E}$. 

Since $T\ddbar_b=\ddbar_bT$, we have 
\[\ddbar_{b,m}:=\ddbar_b:\Omega^{0,\bullet}_m(X,E)\To\Omega^{0,\bullet}_m(X,E).\] 
We also write
\[\ol{\pr}^{*}_b:\Omega^{0,\bullet}(X,E)\To\Omega^{0,\bullet}(X,E)\]
to denote the formal adjoint of $\ddbar_b$ with respect to $(\,\cdot\,|\,\cdot\,)_E$. Since $\langle\,\cdot\,|\,\cdot\,\rangle_E$ and $\langle\,\cdot\,|\,\cdot\,\rangle$ are rigid, we can check that 
\begin{equation}\label{e-gue150517im}
\begin{split}
&T\ddbar^{*}_b=\ddbar^{*}_bT\ \ \mbox{on $\Omega^{0,\bullet}(X,E)$},\\
&\ddbar^{*}_{b,m}:=\ddbar^{*}_b:\Omega^{0,\bullet}_m(X,E)\To\Omega^{0,\bullet}_m(X,E),\ \ \forall m\in\mathbb Z.
\end{split}
\end{equation}
Let $\Box_{b,m}$ denote the $m$-th Kohn Laplacian given by
\begin{equation}\label{e-gue151113ym}
\Box_{b,m}:=(\ddbar_b+\ddbar^*_b)^2:\Omega^{0,\bullet}_m(X,E)\To\Omega^{0,\bullet}_m(X,E).
\end{equation}
We extend $\Box_{b,m}$ to $L^{2}_m(X,T^{*0,\bullet}X\otimes E)$ by 
\begin{equation}\label{e-gue151113ymI}
\Box_{b,m}:{\rm Dom\,}\Box_{b,m}\subset L^{2}_m(X,T^{*0,\bullet}X\otimes E)\To L^{2}_m(X,T^{*0,\bullet}X\otimes E)\,,
\end{equation}
where ${\rm Dom\,}\Box_{b,m}:=\{u\in L^{2}_m(X,T^{*0,\bullet}X\otimes E);\, \Box_{b,m}u\in L^{2}_m(X,T^{*0,\bullet}X\otimes E)\}$, where for any $u\in L^{2}_m(X,T^{*0,\bullet}X\otimes E)$, $\Box_{b,m}u$ is defined in the sense of distribution. It is well-known that $\Box_{b,m}$ is self-adjoint, ${\rm Spec\,}\Box_{b,m}$ is a discrete subset of $[0,\infty[$ and for every $\nu\in{\rm Spec\,}\Box_{b,m}$, $\nu$ is an eigenvalue of $\Box_{b,m}$ (see Section 3 in~\cite{CHT}). Let $e^{-t\Box_{b,m}}$ and $e^{-t\Box_{b,m}}(x,y)\in C^\infty(X\times X, (T^{*0,\bullet}_yX\otimes E_y)\boxtimes(T^{*0,\bullet}_xX\otimes E_x))$ be associated heat operator and heat kernel (see \eqref{e-gue151023a}). Let $N$ be the number operator on $T^{*0,\bullet}X$, i.e. $N$ acts on $T^{*0,q}X$ by multiplication by $q$. By \cite[Theorem 1.7]{CHT}, we have the following asymptotic expansion: 
\begin{equation}\label{E:5.5.10b}
\operatorname{STr}  \lbrack N e^{-t \Box_{b,m}} \rbrack:=\int \operatorname{STr} (Ne^{-t\Box_m} )(x,x)dv_X(x)\sim\sum^\infty_{j=0} \hat{B}_{m,-n+\frac{j}{2}}t^{-n+\frac{j}{2}}\ \ \mbox{as $t\To0^+$}, 
\end{equation}
where $\hat{B}_{m,-n+\frac{j}{2}}\in\Complex$ independent of $t$, for each $j$, and $\operatorname{STr}$ denotes the super trace (see the discussion in the beginning of Section~\ref{s-gue160502q}). We denote by $\Pi^\perp_m:L^2_m(X,T^{0,\bullet}X\otimes E)\To({\rm Ker\,}\Box_{b,m})^\perp$ the orthogonal projection. From \eqref{E:5.5.10b}, for $\operatorname{Re}(z)>n$, we can define 
\begin{equation}\label{E:5.5.12b}
\theta_{b,m}(z)  = - M \left\lbrack \operatorname{STr}  \lbrack N e^{-t \Box_{b,m}} \Pi^\perp_m  \rbrack \right\rbrack  =   - \operatorname{STr} \left\lbrack N ({\Box}_{b,m})^{-z} {\Pi}^\perp_m \right\rbrack
\end{equation}
and $\theta_{b,m}(z)$ extends to a meromorphic function on $\mathbb{C}$ with poles contained in the set  
\[\set{\ell-\frac{j}{2};\, \ell,j\in\mathbb Z},\]
its possible poles are simple, and $\theta_{b,m}(z)$ is holomorphic at $0$ (see Lemma~\ref{L:mero}), where $M$ denotes the Mellin transformation (see Definition~\ref{d-gue160313}). The $m$-th Fourier component of the analytic torsion for the vector bundle $E$ over $X$ is given by $\exp ( -\frac{1}{2} \theta_{b,m}'(0) )$ (see Definition~\ref{d-gue160502w}). 

We denote by $\dot{\mathcal{R}}=\dot{\mathcal{R}}(x)$ the Hermitian matrix $\dot{\mathcal{R}}(x)\in \operatorname{End}(T^{1,0}_xX)$ such that for $V, W \in T^{1,0}_xX$,
\begin{equation}\label{E:1.5.15m}
i d\omega_0(x)(V, \overline{W}) = \langle\,\dot{\mathcal{R}}(x)V\,|\,\overline{W}\,\rangle.
\end{equation}
Let $\det\dot{\mathcal{R}}(x)=\lambda_1(x)\cdots\lambda_{n}(x)$, where $\lambda_1(x),\ldots,\lambda_{n}(x)$ are eigenvalues of $\dot{\mathcal{R}}(x)$. 

Our main result is the following

\begin{theorem}\label{t-gue160502w}
With the notations and assumptions above, as $m \to+\infty$, we have
\begin{equation}\label{E:5.5.33m}
\theta_{b,m}'(0) =  \frac{\operatorname{rk}(E)}{4\pi}
 \int_X \log \left( \det \left(  \frac{m \dot{\mathcal{R}} }{2\pi} \right) \right) e^{-m\frac{d\omega_0}{2\pi}} \wedge (-\omega_0) + o(m^{n}).
\end{equation}
\end{theorem}

Note that the proof of Theorem~\ref{t-gue160502w} is based on Theorem~\ref{T:1.6.1}, Theorem~\ref{t-gue160423} and Theorem~\ref{t-gue160427} which are the main technical results of this paper.


This paper is organized as follows. In Section 2, we collect some notations, definitions and terminology we use throughout. In Section 3, we study the asymptotic behavior of the heat kernel $e^{-\frac{t}{m}\Box_{b,m}}(x, x)$ when $m \to+\infty$ as well as when $t \to 0^+$. In Section 4, we introduce the $m$-th Fourier component of the analytic torsion for the CR manifolds with a transversal CR $S^1$-action. In Section 5, we establish the asymptotic formula for the $m$-th Fourier component of the analytic torsion. 



\section{Preliminaries}\label{s:prelim}

\subsection{Some standard notations}\label{s-gue150508b}
We use the following notations: $\mathbb N=\set{1,2,\ldots}$, $\mathbb N_0=\mathbb N\cup\set{0}$, $\Real$ 
is the set of real numbers, $\Real_+:=\set{x\in\Real;\, x>0}$, $\ol\Real_+:=\set{x\in\Real;\, x\geq0}$.
For a multiindex $\alpha=(\alpha_1,\ldots,\alpha_n)\in\mathbb N_0^n$
we set $\abs{\alpha}=\alpha_1+\cdots+\alpha_n$. For $x=(x_1,\ldots,x_n)$ we write
\[
\begin{split}
&x^\alpha=x_1^{\alpha_1}\ldots x^{\alpha_n}_n,\quad
 \pr_{x_j}=\frac{\pr}{\pr x_j}\,,\quad
\pr^\alpha_x=\pr^{\alpha_1}_{x_1}\ldots\pr^{\alpha_n}_{x_n}=\frac{\pr^{\abs{\alpha}}}{\pr x^\alpha}\,.
\end{split}
\]
Let $z=(z_1,\ldots,z_n)$, $z_j=x_{2j-1}+ix_{2j}$, $j=1,\ldots,n$, be coordinates of $\Complex^n$.
We write
\[
\begin{split}
&z^\alpha=z_1^{\alpha_1}\ldots z^{\alpha_n}_n\,,\quad\ol z^\alpha=\ol z_1^{\alpha_1}\ldots\ol z^{\alpha_n}_n\,,\\
&\pr_{z_j}=\frac{\pr}{\pr z_j}=
\frac{1}{2}\Big(\frac{\pr}{\pr x_{2j-1}}-i\frac{\pr}{\pr x_{2j}}\Big)\,,\quad\pr_{\ol z_j}=
\frac{\pr}{\pr\ol z_j}=\frac{1}{2}\Big(\frac{\pr}{\pr x_{2j-1}}+i\frac{\pr}{\pr x_{2j}}\Big),\\
&\pr^\alpha_z=\pr^{\alpha_1}_{z_1}\ldots\pr^{\alpha_n}_{z_n}=\frac{\pr^{\abs{\alpha}}}{\pr z^\alpha}\,,\quad
\pr^\alpha_{\ol z}=\pr^{\alpha_1}_{\ol z_1}\ldots\pr^{\alpha_n}_{\ol z_n}=
\frac{\pr^{\abs{\alpha}}}{\pr\ol z^\alpha}\,.
\end{split}
\]

Let $X$ be a $C^\infty$ orientable paracompact manifold. 
We let $TX$ and $T^*X$ denote the tangent bundle of $X$ and the cotangent bundle of $X$ respectively.
The complexified tangent bundle of $X$ and the complexified cotangent bundle of $X$ 
will be denoted by $\Complex TX$ and $\Complex T^*X$ respectively. We write $\langle\,\cdot\,,\cdot\,\rangle$ 
to denote the pointwise duality between $T^*X$ and $TX$.
We extend $\langle\,\cdot\,,\cdot\,\rangle$ bilinearly to $\Complex T^*X\times\Complex TX$. For $u\in \Complex T^*X$, $v\in\Complex TX$, we also write $u(v):=\langle\,u\,,v\,\rangle$.

Let $E$ be a $C^\infty$ vector bundle over $X$. The fiber of $E$ at $x\in X$ will be denoted by $E_x$.
Let $F$ be another vector bundle over $X$. We write 
$E\boxtimes F$ to denote the vector bundle over $X\times X$ with fiber over $(x, y)\in X\times X$ 
consisting of the linear maps from $E_x$ to $F_y$.  

Let $Y\subset X$ be an open set. The spaces of
smooth sections of $E$ over $Y$ and distribution sections of $E$ over $Y$ will be denoted by $C^\infty(Y, E)$ and $D'(Y, E)$ respectively.
Let $E'(Y, E)$ be the subspace of $D'(Y, E)$ whose elements have compact support in $Y$.
For $m\in\Real$, we let $H^m(Y, E)$ denote the Sobolev space
of order $m$ of sections of $E$ over $Y$. Put
\begin{gather*}
H^m_{\rm loc\,}(Y, E)=\big\{u\in D'(Y, E);\, \varphi u\in H^m(Y, E),
      \,\forall \varphi\in C^\infty_0(Y)\big\}\,,\\
      H^m_{\rm comp\,}(Y, E)=H^m_{\rm loc}(Y, E)\cap E'(Y, E)\,.
\end{gather*}

\subsection{Set up and terminology}\label{s-gue150508bI} 

Let $(X, T^{1,0}X)$ be a compact CR manifold of dimension $2n+1$, $n\geq 1$, where $T^{1,0}X$ is a CR structure of $X$. That is $T^{1,0}X$ is a subbundle of rank $n$ of the complexified tangent bundle $\mathbb{C}TX$, satisfying $T^{1,0}X\cap T^{0,1}X=\{0\}$, where $T^{0,1}X=\overline{T^{1,0}X}$, and $[\mathcal V,\mathcal V]\subset\mathcal V$, where $\mathcal V=C^\infty(X, T^{1,0}X)$. We assume that $X$ admits a $S^1$ action: $S^1\times X\rightarrow X$. We write $e^{i\theta}$ to denote the $S^1$ action. Let $T\in C^\infty(X, TX)$ be the global real vector field induced by the $S^1$ action given by 
$(Tu)(x)=\frac{\partial}{\partial\theta}\left(u(e^{i\theta}\circ x)\right)|_{\theta=0}$, $u\in C^\infty(X)$. 

\begin{definition}\label{d-gue160502}
We say that the $S^1$ action $e^{i\theta}$ is CR if
$[T, C^\infty(X, T^{1,0}X)]\subset C^\infty(X, T^{1,0}X)$ and the $S^1$ action is transversal if for each $x\in X$,
$\Complex T(x)\oplus T_x^{1,0}X\oplus T_x^{0,1}X=\mathbb CT_xX$. Moreover, we say that the $S^1$ action is locally free if $T\neq0$ everywhere. 
\end{definition}

We assume throughout that $(X, T^{1,0}X)$ is a compact connected CR manifold with a transversal CR locally free $S^1$ action $e^{i\theta}$ and we let $T$ be the global vector field induced by the $S^1$ action. Let $\omega_0\in C^\infty(X,T^*X)$ be the global real one form determined by $\langle\,\omega_0\,,\,u\,\rangle=0$, for every $u\in T^{1,0}X\oplus T^{0,1}X$ and $\langle\,\omega_0\,,\,T\,\rangle=-1$. 
Assume $X=X_{p_1}\bigcup X_{p_2}\bigcup\cdots\bigcup X_{p_k}$ (see \eqref{e-gue150802bm}), $p=:p_1<p_2<\cdots<p_k$. In this work, we assume that $p_1=1$ 
and we denote $X_{{\rm reg\,}}:=X_{p_1}=X_1$. 

\begin{definition}\label{d-gue150508f}
For $p\in X$, the Levi form $\mathcal L_p$ is the Hermitian quadratic form on $T^{1,0}_pX$ given by 
$\mathcal{L}_p(U,\ol V)=-\frac{1}{2i}\langle\,d\omega_0(p)\,,\,U\wedge\ol V\,\rangle$, $U, V\in T^{1,0}_pX$.
\end{definition}

If the Levi form $\mathcal{L}_p$ is positive definite, we say that $X$ is strongly pseudoconvex at $p$. If the Levi form is positive definite at every point of $X$, we say that $X$ is strongly pseudoconvex. We assume throughout that $(X, T^{1,0}X)$ is strongly pseudoconvex. 

Denote by $T^{*1,0}X$ and $T^{*0,1}X$ the dual bundles of
$T^{1,0}X$ and $T^{0,1}X$ respectively. Define the vector bundle of $(0,q)$ forms by
$T^{*0,q}X=\Lambda^q(T^{*0,1}X)$. Put $T^{*0,\bullet}X:=\oplus_{j\in\set{0,1,\ldots,n}}T^{*0,j}X$.
Let $D\subset X$ be an open subset. Let $\Omega^{0,q}(D)$
denote the space of smooth sections of $T^{*0,q}X$ over $D$ and let $\Omega_0^{0,q}(D)$
be the subspace of $\Omega^{0,q}(D)$ whose elements have compact support in $D$. Put 
\[\begin{split}
&\Omega^{0,\bullet}(D):=\oplus_{j\in\set{0,1,\ldots,n}}\Omega^{0,j}(D),\\
&\Omega^{0,\bullet}_0(D):=\oplus_{j\in\set{0,1,\ldots,n}}\Omega^{0,j}_0(D).\end{split}\]
Similarly, if $E$ is a vector bundle over $D$, then we let $\Omega^{0,q}(D,E)$ denote the space of smooth sections of $T^{*0,q}X\otimes E$ over $D$ and let $\Omega_0^{0,q}(D,E)$ be the subspace of $\Omega^{0,q}(D, E)$ whose elements have compact support in $D$. Put 
 \[\begin{split}
&\Omega^{0,\bullet}(D,E):=\oplus_{j\in\set{0,1,\ldots,n}}\Omega^{0,j}(D,E),\\
&\Omega^{0,\bullet}_0(D,E):=\oplus_{j\in\set{0,1,\ldots,n}}\Omega^{0,j}_0(D,E).\end{split}\]

Fix $\theta_0\in]-\pi, \pi[$, $\theta_0$ small. Let
$$d e^{i\theta_0}: \mathbb CT_x X\rightarrow \mathbb CT_{e^{i\theta_0}x}X$$
denote the differential map of $e^{i\theta_0}: X\rightarrow X$. By the CR property of the $S^1$ action, we can check that
\begin{equation}\label{e-gue150508fa}
\begin{split}
de^{i\theta_0}:T_x^{1,0}X\rightarrow T^{1,0}_{e^{i\theta_0}x}X,\\
de^{i\theta_0}:T_x^{0,1}X\rightarrow T^{0,1}_{e^{i\theta_0}x}X,\\
de^{i\theta_0}(T(x))=T(e^{i\theta_0}x).
\end{split}
\end{equation}
Let $(e^{i\theta_0})^*:\Lambda^r(\Complex T^*X)\To\Lambda^r(\Complex T^*X)$ be the pull-back map by $e^{i\theta_0}$, $r=0,1,\ldots,2n+1$. From \eqref{e-gue150508fa}, it is easy to see that for every $q=0,1,\ldots,n$, 
\begin{equation}\label{e-gue150508faI}
(e^{i\theta_0})^*:T^{*0,q}_{e^{i\theta_0}x}X\To T^{*0,q}_{x}X.
\end{equation}
Let $u\in\Omega^{0,q}(X)$. Define
\begin{equation}\label{e-gue150508faII}
Tu:=\frac{\pr}{\pr\theta}\bigr((e^{i\theta})^*u\bigr)|_{\theta=0}\in\Omega^{0,q}(X).
\end{equation}
(See also \eqref{lI}.) For every $\theta\in\Real$ and every $u\in C^\infty(X,\Lambda^r(\Complex T^*X))$, we write $u(e^{i\theta}\circ x):=(e^{i\theta})^*u(x)$. It is clear that for every $u\in C^\infty(X,\Lambda^r(\Complex T^*X))$, we have 
\begin{equation}\label{e-gue150510f}
u(x)=\sum_{m\in\mathbb Z}\frac{1}{2\pi}\int^{\pi}_{-\pi}u(e^{i\theta}\circ x)e^{-im\theta}d\theta.
\end{equation}

Let $\ddbar_b:\Omega^{0,q}(X)\rightarrow\Omega^{0,q+1}(X)$ be the tangential Cauchy-Riemann operator. From the CR property of the $S^1$ action, it is straightforward to see that (see also \eqref{e-gue150514f})
\[T\ddbar_b=\ddbar_bT\ \ \mbox{on $\Omega^{0,\bullet}(X)$}.\]

\begin{definition}\label{d-gue50508d}
Let $D\subset U$ be an open set. We say that a function $u\in C^\infty(D)$ is rigid if $Tu=0$. We say that a function $u\in C^\infty(X)$ is Cauchy-Riemann (CR for short)
if $\ddbar_bu=0$. We call $u$ a rigid CR function if  $\ddbar_bu=0$ and $Tu=0$.
\end{definition}

\begin{definition} \label{d-gue150508dI}
Let $F$ be a complex vector bundle over $X$. We say that $F$ is rigid (CR) if 
$X$ can be covered with open sets $U_j$ with trivializing frames $\set{f^1_j,f^2_j,\dots,f^r_j}$, $j=1,2,\ldots$, such that the corresponding transition matrices are rigid (CR). The frames $\set{f^1_j,f^2_j,\dots,f^r_j}$, $j=1,2,\ldots$, are called rigid (CR) frames. 
\end{definition}

\begin{definition}\label{d-gue150514f}
Let $F$ be a complex rigid vector bundle over $X$ and let $\langle\,\cdot\,|\,\cdot\,\rangle_F$ be a Hermitian metric on $F$. We say that $\langle\,\cdot\,|\,\cdot\,\rangle_F$ is a rigid Hermitian metric if for every rigid local frames $f_1,\ldots, f_r$ of $F$, we have $T\langle\,f_j\,|\,f_k\,\rangle_F=0$, for every $j,k=1,2,\ldots,r$. 
\end{definition} 

It is well-known that there is a rigid Hermitian metric on any rigid vector bundle $F$ (see Theorem 2.10 in~\cite{CHT} and Theorem 10.5 in~\cite{Hsiao14}). Note that  Baouendi-Rothschild-Treves~\cite{BRT85} proved that $T^{1,0}X$ is a rigid complex vector bundle over $X$.

From now on, let $E$ be a rigid CR vector bundle over $X$ and we take a rigid Hermitian metric $\langle\,\cdot\,|\,\cdot\,\rangle_E$ on $E$ and take a rigid Hermitian metric $\langle\,\cdot\,|\,\cdot\,\rangle$ on $\Complex TX$ such that $T^{1,0}X\perp T^{0,1}X$, $T\perp (T^{1,0}X\oplus T^{0,1}X)$, $\langle\,T\,|\,T\,\rangle=1$.
The Hermitian metrics on $E$ and $\Complex TX$ induce Hermitian metrics $\langle\,\cdot\,|\,\cdot\,\rangle$ and $\langle\,\cdot\,|\,\cdot\,\rangle_E$ on $T^{*0,\bullet}X$ and $T^{*0,\bullet}X\otimes E$ respectively. Let $A(x,y)\in (T^{*,\bullet}_yX\otimes E_y)\boxtimes (T^{*,\bullet}_xX\otimes E_x)$. We write $\abs{A(x,y)}$ to denote the natural matrix norm of $A(x,y)$ induced by $\langle\,\cdot\,|\,\cdot\,\rangle_{E}$.
We denote by $dv_X=dv_X(x)$ the volume form on $X$ induced by the fixed 
Hermitian metric $\langle\,\cdot\,|\,\cdot\,\rangle$ on $\Complex TX$. Then we get natural global $L^2$ inner products $(\,\cdot\,|\,\cdot\,)_{E}$, $(\,\cdot\,|\,\cdot\,)$
on $\Omega^{0,\bullet}(X,E)$ and $\Omega^{0,\bullet}(X)$ respectively. We denote by $L^2(X,T^{*0,q}X
\otimes E)$ and $L^2(X,T^{*0,q}X)$ the completions of $\Omega^{0,q}(X,E)$ and $\Omega^{0,q}(X)$ with respect to $(\,\cdot\,|\,\cdot\,)_{E}$ and $(\,\cdot\,|\,\cdot\,)$ respectively. Similarly, we denote by $L^2(X,T^{*0,\bullet}X
\otimes E)$ and $L^2(X,T^{*0,\bullet}X)$ the completions of $\Omega^{0,\bullet}(X,E)$ and $\Omega^{0,\bullet}(X)$ with respect to $(\,\cdot\,|\,\cdot\,)_{E}$ and $(\,\cdot\,|\,\cdot\,)$ respectively. We extend $(\,\cdot\,|\,\cdot\,)_{E}$ and $(\,\cdot\,|\,\cdot\,)$ to $L^2(X,T^{*0,\bullet}X\otimes E)$ and $L^2(X,T^{*0,\bullet}X)$ in the standard way respectively.  For $f\in L^{2}(X,T^{*0,\bullet}X\otimes E)$, we denote $\norm{f}^2_{E}:=(\,f\,|\,f\,)_{E}$. Similarly, for $f\in L^{2}(X,T^{*0,\bullet}X)$, we denote $\norm{f}^2:=(\,f\,|\,f\,)$. 

We also write $\ddbar_b$ to denote the tangential Cauchy-Riemann operator acting on forms with values in $E$:
\[\ddbar_b:\Omega^{0,\bullet}(X, E)\To\Omega^{0,\bullet}(X,E).\]
Since $E$ is rigid, we can also define $Tu$ for every $u\in\Omega^{0,q}(X,E)$ and we have 
\begin{equation}\label{e-gue150508d}
T\ddbar_b=\ddbar_bT\ \ \mbox{on $\Omega^{0,\bullet}(X,E)$}.
\end{equation}
For every $m\in\mathbb Z$, let
\begin{equation}\label{e-gue150508dI}
\begin{split}
&\Omega^{0,q}_m(X,E):=\set{u\in\Omega^{0,q}(X,E);\, Tu=imu},\ \ q=0,1,2,\ldots,n,\\
&\Omega^{0,\bullet}_m(X,E):=\set{u\in\Omega^{0,\bullet}(X,E);\, Tu=imu}.
\end{split}
\end{equation} 
For each $m\in\mathbb Z$, we denote by $L^2_m(X,T^{*0,q}X\otimes E)$ and $L^2_m(X,T^{*0,q}X)$ the completions of $\Omega^{0,q}_m(X,E)$ and $\Omega^{0,q}_m(X)$ with respect to $(\,\cdot\,|\,\cdot\,)_{E}$ and $(\,\cdot\,|\,\cdot\,)$ respectively. Similarly, we denote by $L^2_m(X,T^{*0,\bullet}X\otimes E)$ and $L^2_m(X,T^{*0,\bullet}X)$ the completions of $\Omega^{0,\bullet}_m(X,E)$ and $\Omega^{0,\bullet}_m(X)$ with respect to $(\,\cdot\,|\,\cdot\,)_{E}$ and $(\,\cdot\,|\,\cdot\,)$ respectively.

\section{Asymptotic expansion of heat kernels}

\subsection{Heat kernels of the Kohn Laplacians}

Since $T\ddbar_b=\ddbar_bT$, we have 
\[\ddbar_{b,m}:=\ddbar_b:\Omega^{0,\bullet}_m(X,E)\To\Omega^{0,\bullet}_m(X,E),\ \ \forall m\in\mathbb Z.\] 
We also write
\[\ol{\pr}^{*}_b:\Omega^{0,\bullet}(X,E)\To\Omega^{0,\bullet}(X,E)\]
to denote the formal adjoint of $\ddbar_b$ with respect to $(\,\cdot\,|\,\cdot\,)_E$. Since $\langle\,\cdot\,|\,\cdot\,\rangle_E$ and $\langle\,\cdot\,|\,\cdot\,\rangle$ are rigid, we can check that 
\begin{equation}\label{e-gue150517i}
\begin{split}
&T\ddbar^{*}_b=\ddbar^{*}_bT\ \ \mbox{on $\Omega^{0,\bullet}(X,E)$},\\
&\ddbar^{*}_{b,m}:=\ddbar^{*}_b:\Omega^{0,\bullet}_m(X,E)\To\Omega^{0,\bullet}_m(X,E),\ \ \forall m\in\mathbb Z.
\end{split}
\end{equation}
Now, we fix $m\in\mathbb Z$. The $m$-th Fourier component of Kohn Laplacian is given by
\begin{equation}\label{e-gue151113y}
\Box_{b,m}:=(\ddbar_{b,m}+\ddbar^*_{b,m})^2:\Omega^{0,\bullet}_m(X,E)\To\Omega^{0,\bullet}_m(X,E).
\end{equation}
We extend $\Box_{b,m}$ to $L^{2}_m(X,T^{*0,\bullet}X\otimes E)$ by 
\begin{equation}\label{e-gue151113yI}
\Box_{b,m}:{\rm Dom\,}\Box_{b,m}\subset L^{2}_m(X,T^{*0,\bullet}X\otimes E)\To L^{2}_m(X,T^{*0,\bullet}X\otimes E)\,,
\end{equation}
where ${\rm Dom\,}\Box_{b,m}:=\{u\in L^{2}_m(X,T^{*0,\bullet}X\otimes E);\, \Box_{b,m}u\in L^{2}_m(X,T^{*0,\bullet}X\otimes E)\}$, where for any $u\in L^{2}_m(X,T^{*0,\bullet}X\otimes E)$, $\Box_{b,m}u$ is defined in the sense of distribution. 
It is well-known that $\Box_{b,m}$ is self-adjoint, ${\rm Spec\,}\Box_{b,m}$ is a discrete subset of $[0,\infty[$ and for every $\nu\in{\rm Spec\,}\Box_{b,m}$, $\nu$ is an eigenvalue of $\Box_{b,m}$ (see Section 3 in~\cite{CHT}).
For every $\nu\in{\rm Spec\,}\Box_{b,m}$, let $\set{f^\nu_1,\ldots,f^\nu_{d_{\nu}}}$ be an orthonormal frame for the eigenspace of $\Box_{b,m}$ with eigenvalue $\nu$. The heat kernel $e^{-t\Box_{b,m}}(x,y)$ is given by 
\begin{equation}\label{e-gue151023a}
e^{-t\Box_{b,m}}(x,y)=\sum_{\nu\in{\rm Spec\,}\Box_{b,m}}\sum^{d_{\nu}}_{j=1}e^{-\nu t}f^\nu_j(x)\otimes(f^\nu_j(y))^\dagger,
\end{equation}
where $f^\nu_j(x)\otimes(f^\nu_j(y))^\dagger$ denotes the linear map: 
\[\begin{split}
f^\nu_j(x)\otimes(f^\nu_j(y))^\dagger:T^{*0,\bullet}_yX\otimes E_y&\To T^{*0,\bullet}_xX\otimes E_x,\\
u(y)\in T^{*0,\bullet}_yX\otimes E_y&\To f^\nu_j(x)\langle\,u(y)\,|\,f^\nu_j(y)\,\rangle_E\in T^{*0,\bullet}_xX\otimes E_x.\end{split}\]
Let $e^{-t\Box_{b,m}}:L^2(X,T^{*0,\bullet}X\otimes E)\To L^2_m(X,T^{*0,\bullet}X\otimes E)$ be the continuous operator with distribution kernel $e^{-t\Box_{b,m}}(x,y)$.

We denote by $\dot{\mathcal{R}}$ the Hermitian matrix $\dot{\mathcal{R}} \in \operatorname{End}(T^{1,0}X)$ such that for $V, W \in T^{1,0}X$,
\begin{equation}\label{E:1.5.15}
i d\omega_0(V, \overline{W}) = \langle\,\dot{\mathcal{R}} V\,|\,\overline{W}\,\rangle.
\end{equation}
Let $\{ \omega_j \}_{j=1}^n$ be a local orthonormal frame of $T^{1,0}X$ with dual frame $\{ \omega^j \}_{j=1}^n$.
Set 
\begin{equation}
\gamma_d = - i\sum^n_{l,j=1} d\omega_0(\omega_j, \overline{\omega}_l) \overline{\omega}^l \wedge \iota_{\overline{\omega}_j},
\end{equation}
where $\iota_{\overline{\omega}_j}$ denotes the interior product of $\overline{\omega}_j$.
Then $\gamma_d \in \operatorname{End}(T^{*0,\bullet}X)$ and $-id\omega_0$ acts as the derivative $\gamma_d$ on $T^{*0,\bullet}X$. 
If we choose $\{ \omega_j \}_{j=1}^n$ to be an orthonormal basis of $T^{1,0}X$ such that
\begin{equation}\label{e-gue160127b}
\dot{\mathcal{R}}(x) = \operatorname{diag} (a_1(x), \cdots, a_n(x)) \in \operatorname{End}(T_x^{1,0}X),
\end{equation}\
then
\begin{equation}\label{E:1.5.19}
\gamma_d(x) = -\sum^n_{j=1} a_j(x) \overline{\omega}^j \wedge \iota_{\overline{\omega}_j}.
\end{equation}
Define ${\rm det\,}\dot{\mathcal{R}}(x):=a_1(x)\cdots a_n(x)$. 

We pause and introduce some notations. Let $F$ be a vector bundle over a smooth manifold $M$ and let $H(x,m), G(x,m)\in C^\infty(M,F)$ be $m$-dependent smooth functions. 
We write $H(x,m)=G(x,m)+o(m^p)$ in $C^\ell(M,F)$ locally uniformly on $M$, where $p\in\Real$, $\ell\in\mathbb N_0$, if for any compact subset $K\Subset M$ and every $\epsilon>0$, there is a $m_0>0$ independent of $m$ such that for all $m\geq m_0$ and for all $x\in K$, we have
\begin{equation}\label{e-gue160127}
\abs{H(x,m)-G(x,m)}_{C^\ell(K,F)}\leq \epsilon m^p.
\end{equation}
When $M$ is compact, we omit locally.

Similarly, we write $H(x,m)=G(x,m)+O(m^p)$ in $C^\ell(M,F)$ locally uniformly on $M$, where $p\in\Real$, $\ell\in\mathbb N_0$, if for any compact subset $K\Subset M$, there are $m_0>0$ and $C_K>0$ independent of $m$ such that for all $m\geq m_0$ and for all $x\in K$, we have
\begin{equation}\label{e-gue160127I}
\abs{H(x,m)-G(x,m)}_{C^\ell(K,F)}\leq C_Km^p.
\end{equation}
When $M$ is compact, we omit locally.

Let $B(t,x)\in C^\infty(\Real_+\times M,F)$. 
We write $B(t,x)=\sum^k_{j=-p}B_j(x)t^j+O(t^{k_1})$ in $C^\ell(M,F)$ locally uniformly on $M$, where $k, k_1, p\in\mathbb Z$ and $B_j(x)\in C^\infty(M,F)$, $j=-p,-p+1,\ldots,k$, if for any compact subset $K\Subset M$, there are constants $C_K>0$ and $\varepsilon>0$ independent of $t$ such that
\begin{equation}\label{e-gue160126gI}
\abs{B(t,x)-\sum^k_{j=-p}B_j(x)t^j}_{C^\ell(K,F)}\leq C_Kt^{k_1},\ \ \forall 0<t<\varepsilon.
\end{equation}
When $M$ is compact, we omit locally. 

We return to our situation. Let $A_\ell\in C^\infty(X, \operatorname{End}(T^{*0,\bullet}X))$, $\ell=-n,-n+1,\ldots$, such that for every $k\in\mathbb Z$, 
\begin{equation}\label{E:5.5.34}
(2\pi)^{-n-1} \frac{\det(\dot{\mathcal{R}}) \exp(t \gamma_d)}{\det(1-\exp(-t\dot{\mathcal{R}}))}(x)=\sum_{\ell=-n}^k A_\ell(x)t^\ell +O(t^{k+1})
\end{equation}
in $C^0(X,\operatorname{End}(T^{*0,\bullet}X))$ uniformly on $X$.

Fix $x, y\in X$. Let $d(x,y)$ denote the standard Riemannian distance of $x$ and $y$ with respect to the given Hermitian metric. Put $d(x,X_{{\rm sing\,}}):=\inf\set{d(x,y);\, y\in X_{{\rm sing\,}}}$.

The first goal of this section is to prove the following

\begin{theorem}\label{T:1.6.1}
With the notations used above, fix $I\Subset\Real_+$. 
For every $\epsilon>0$, there are $m_0>0$, $\varepsilon_0>0$ and $C>0$ such that for all $m\geq m_0$, we have
\begin{equation}\label{E:1.6.4}
\begin{split}
&\Big|e^{-\frac{t}{m} \Box_{b,m}}(x,x)-(2\pi)^{-n-1} \frac{\det(\dot{\mathcal{R}}) \exp(t \gamma_d)}{\det(1-\exp(-t\dot{\mathcal{R}}))}(x)  \otimes \operatorname{Id}_{E_x}m^n\Big|\\
&\leq \epsilon m^n+Cm^{n}e^{-\varepsilon_0md(x,X_{{\rm sing\,}})^2},\ \ \forall (t,x)\in I\times X_{{\rm reg\,}}.
\end{split}
\end{equation}
Here we use the convention that if an eigenvalue $a_j(x)$ of $\dot{\mathcal{R}}(x)$ is zero, then its contribution for $\frac{\det(\dot{\mathcal{R}})}{\det(1-\exp(-t\dot{\mathcal{R}}))}(x)$ is $\frac{1}{t}$ (see \eqref{e-gue160127b}). 

\end{theorem}

\subsection{BRT trivializations}

To prove Theorem~\ref{T:1.6.1}, we need some preparations. We first need the following result due to Baouendi-Rothschild-Treves~\cite{BRT85}.

\begin{theorem}\label{t-gue150514}
For every point $x_0\in X$, we can find local coordinates $x=(x_1,\cdots,x_{2n+1})=(z,\theta)=(z_1,\cdots,z_{n},\theta), z_j=x_{2j-1}+ix_{2j},j=1,\cdots,n, x_{2n+1}=\theta$, defined in some small neighborhood $D=\{(z, \theta): \abs{z}<\delta, -\varepsilon_0<\theta<\varepsilon_0\}$ of $x_0$, $\delta>0$, $0<\varepsilon_0<\pi$, such that $(z(x_0),\theta(x_0))=(0,0)$ and 
\begin{equation}\label{e-can}
\begin{split}
&T=\frac{\partial}{\partial\theta}\\
&Z_j=\frac{\partial}{\partial z_j}+i\frac{\partial\varphi}{\partial z_j}(z)\frac{\partial}{\partial\theta},j=1,\cdots,n
\end{split}
\end{equation}
where $Z_j(x), j=1,\cdots, n$, form a basis of $T_x^{1,0}X$, for each $x\in D$ and $\varphi(z)\in C^\infty(D,\mathbb R)$ independent of $\theta$. We call $(D,(z,\theta),\varphi)$ BRT trivialization.
\end{theorem}

By using BRT trivialization, we get another way to define $Tu, \forall u\in\Omega^{0,q}(X)$. Let $(D,(z,\theta),\varphi)$ be a BRT trivialization. It is clear that
$$\{d\overline{z_{j_1}}\wedge\cdots\wedge d\overline{z_{j_q}}, 1\leq j_1<\cdots<j_q\leq n\}$$
is a basis for $T^{\ast0,q}_xX$, for every $x\in D$. Let $u\in\Omega^{0,q}(X)$. On $D$, we write
\begin{equation}\label{e-gue150524fb}
u=\sum\limits_{1\leq j_1<\cdots<j_q\leq n}u_{j_1\cdots j_q}d\overline{z_{j_1}}\wedge\cdots\wedge d\overline{z_{j_q}}.
\end{equation}
Then on $D$ we can check that
\begin{equation}\label{lI}
Tu=\sum\limits_{1\leq j_1<\cdots<j_q\leq n}(Tu_{j_1\cdots j_q})d\overline{z_{j_1}}\wedge\cdots\wedge d\overline{z_{j_q}}
\end{equation}
and $Tu$ is independent of the choice of BRT trivializations. Note that on BRT trivialization $(D,(z,\theta),\varphi)$, we have 
\begin{equation}\label{e-gue150514f}
\ddbar_b=\sum^n_{j=1}d\ol z_j\wedge(\frac{\partial}{\partial\ol z_j}-i\frac{\partial\varphi}{\partial\ol z_j}(z)\frac{\partial}{\partial\theta}).
\end{equation} 

\subsection{Local heat kernels on BRT trivializations}

Until further notice, we fix $m\in\mathbb Z$. 
Let $B:=(D,(z,\theta),\varphi)$ be a BRT trivialization. We may assume that $D=U\times]-\varepsilon,\varepsilon[$, where $\varepsilon>0$ and $U$ is an open set of $\Complex^n$. Since $E$ is rigid, we can consider $E$ as a holomorphic vector bundle over $U$. We may assume that $E$ is trivial on $U$. Consider $L\To U$ be a trivial line bundle with non-trivial Hermitian fiber metric $\abs{1}^2_{h^L}=e^{-2\varphi}$. Let $(L^m,h^{L^m})\To U$ be the $m$-th power of $(L,h^L)$.  Let $\Omega^{0,q}(U,E\otimes L^m)$ and $\Omega^{0,q}(U,E)$ be the spaces of $(0,q)$ forms on $U$ with values in $E\otimes L^m$ and $E$ respectively, $q=0,1,2,\ldots,n$. Put 
\[
\begin{split}
&\Omega^{0,\bullet}(U,E\otimes L^m):=\oplus_{j\in\set{0,1,\ldots,n}}\Omega^{0,j}(U,E\otimes L^m),\\
&\Omega^{0,\bullet}(U,E):=\oplus_{j\in\set{0,1,\ldots,n}}\Omega^{0,j}(U,E).
\end{split}\]
Since $L$ is trivial, from now on, we identify $\Omega^{0,\bullet}(U,E)$ with $\Omega^{0,\bullet}(U,E\otimes L^m)$.
Since the Hermitian fiber metric $\langle\,\cdot\,|\,\cdot\,\rangle_E$ is rigid, we can consider $\langle\,\cdot\,|\,\cdot\,\rangle_E$ as a Hermitian fiber metric on the holomorphic vector bundle $E$ over $U$. 
Let $\langle\,\cdot\,,\,\cdot\,\rangle$ be the Hermitian metric on $\Complex TU$ given by 
\[\langle\,\frac{\pr}{\pr z_j}\,,\,\frac{\pr}{\pr z_k}\,\rangle=\langle\,\frac{\pr}{\pr z_j}+i\frac{\pr\varphi}{\pr z_j}(z)\frac{\pr}{\pr\theta}\,|\,\frac{\pr}{\pr z_k}+i\frac{\pr\varphi}{\pr z_k}(z)\frac{\pr}{\pr\theta}\,\rangle,\ \ j,k=1,2,\ldots,n.\]
$\langle\,\cdot\,,\,\cdot\,\rangle$ induces a Hermitian metric on $T^{*0,\bullet}U:=\oplus_{j=0}^nT^{*0,j}U$, where  $T^{*0,j}U$ is the bundle of $(0,j)$ forms on $U$, $j=0,1,\ldots,n$. We shall also denote the Hermitian metric by $\langle\,\cdot\,,\,\cdot\,\rangle$. The Hermitian metric on $T^{*0,\bullet}U$ and $E$ induce a
Hermitian metric on $T^{*0,\bullet}U\otimes E$. We shall also denote this induced metric by $\langle\,\cdot\,|\,\cdot\,\rangle_{E}$.
Let $(\,\cdot\,,\,\cdot\,)$ be the $L^2$ inner product on $\Omega^{0,\bullet}(U,E)$ induced by $\langle\,\cdot\,,\,\cdot\,\rangle$, $\langle\,\cdot\,|\,\cdot\,\rangle_E$. Similarly, let $(\,\cdot\,,\,\cdot\,)_m$ be the $L^2$ inner product on $\Omega^{0,\bullet}(U,E\otimes L^m)$ induced by $\langle\,\cdot\,,\,\cdot\,\rangle$, $\langle\,\cdot\,|\,\cdot\,\rangle_E$ and $h^{L^m}$. 

The curvature of $L$ induced by $h^L$ is given by $R^L:=2\pr\ddbar\varphi$. Let $\dot{R}^L\in \operatorname{End}(T^{1,0}U)$ be the Hermitian matrix given by 
\[R^L(W,\ol Y)=\langle\,\dot{R}^LW\,,\,\ol Y\,\rangle,\ \ W, Y\in T^{1,0}U.\]
Let $\{w_j \}_{j=1}^n$ be a local orthonormal frame of $T^{1,0}U$ with dual frame $\{ w^j \}_{j=1}^n$.
Set 
\begin{equation}\label{e-gue160128a}
\omega_d = - \sum_{l,j}R^L(w_j, \overline{w}_l) \overline{w}^l \wedge \iota_{\overline{w}_j},
\end{equation}
where $\iota_{\overline{w}_j}$ denotes the interior product of $\overline{w}_j$. Let $\mathcal{A}_\ell \in C^\infty(U, \operatorname{End}(T^{*0,\bullet}U))$, $\ell=-n,-n+1,\ldots$, such that for every $k\in\mathbb Z$, 
\begin{equation}\label{e-gue160128aI}
(2\pi)^{-n} \frac{\det(\dot{R}^L) \exp(t \omega_d)}{\det(1-\exp(-t\dot{R}^L))}(z) = \sum_{\ell=-n}^k\mathcal{A}_\ell(z)t^\ell+O(t^{k+1})
\end{equation}
in $C^0(U, \operatorname{End}(T^{*0,\bullet}U))$ locally uniformly on $U$. Let 
\[\begin{split}
\pi:D&\To U,\\
(z,\theta)&\To z
\end{split}\]
be the natural projection.
It is straightforward to check that 
\begin{equation}\label{e-gue160129I}
(2\pi)^{-n} \frac{\det(\dot{R}^L)\exp(t\omega_d)}{\det(1-\exp(-t\dot{R}^L))}(\pi(x)) \otimes \operatorname{Id}_{E_{\pi(x)}} = (2\pi)^{-n} \frac{\det(\dot{\mathcal{R}}) \exp(t \gamma_d)}{\det(1-\exp(-t\dot{\mathcal{R}}))}(x)  \otimes \operatorname{Id}_{E_{x}},\ \ \forall x\in D.
\end{equation}
 From \eqref{E:5.5.34}, \eqref{e-gue160128aI} and \eqref{e-gue160129I}, it is easy to see that 
\begin{equation}\label{e-gue160131}
\mathcal{A}_\ell(\pi(x))=(2\pi)A_\ell(x),\ \ \forall x\in D,\ \ \ell=-n,-n+1,\ldots. 
\end{equation}

Let
\[\ddbar:\Omega^{0,\bullet}(U,E\otimes L^m)\To\Omega^{0,\bullet}(U,E\otimes L^m)\]
be the Cauchy-Riemann operator and let
\[\ol{\pr}^{*,m}:\Omega^{0,\bullet}(U,E\otimes L^m)\To\Omega^{0,\bullet}(U,E\otimes L^m)\] 
be the formal adjoint of $\ddbar$ with respect to $(\,\cdot\,,\,\cdot\,)_m$. Put 
\begin{equation}\label{e-gue150606II}
\Box_{B,m}:=(\ddbar+\ol{\pr}^{*,m})^2: \Omega^{0,\bullet}(U,E\otimes L^m)\To\Omega^{0,\bullet}(U,E\otimes L^m).
\end{equation}
We need the following result which is well-known (see Lemma 5.1 in~\cite{CHT}) 

\begin{lemma}\label{l-gue150606}
Let $u\in\Omega^{0,\bullet}_m(X,E)$. On $D$, we write $u(z,\theta)=e^{im\theta}\Td u(z)$, $\Td u(z)\in\Omega^{0,\bullet}(U,E)$. Then,
\begin{equation}\label{e-gue150606III}
e^{-m\varphi}\Box_{B,m}(e^{m\varphi}\Td u)=e^{-im\theta}\Box_{b,m}(u).
\end{equation}
\end{lemma}

Let $z, w\in U$ and let $T(z,w)\in (T^{*0,\bullet}_wU\otimes E_w)\boxtimes(T^{*0,\bullet}_zU\otimes E_z)$. We write $\abs{T(z,w)}$ to denote the standard pointwise matrix norm of $T(z,w)$ induced by $\langle\,\cdot\,|\,\cdot\,\rangle_E$. Let $\Omega^{0,\bullet}_0(U,E)$ be the subspace of $\Omega^{0,\bullet}(U,E)$ whose elements have compact support in $U$. Let $dv_U$ be the volume form on $U$ induced by $\langle\,\cdot\,,\,\cdot\,\rangle$. Assume $T(z,w)\in C^\infty(U\times U,(T^{*0,\bullet}_wU\otimes E_w)\boxtimes(T^{*0,\bullet}_zU\otimes E_z))$. Let $u\in\Omega^{0,\bullet}_0(U,E)$. We define the integral $\int T(z,w)u(w)dv_U(w)$ in the standard way. Let $G(t,z,w)\in C^\infty(\Real_+\times U\times U,(T^{*0,\bullet}_wU\otimes E_w)\boxtimes(T^{*0,\bullet}_zU\otimes E_z))$. We write $G(t)$ to denote the continuous operator
\[\begin{split}
G(t):\Omega^{0,\bullet}_0(U,E)&\To\Omega^{0,\bullet}(U,E),\\
u&\To\int G(t,z,w)u(w)dv_U(w)\end{split}\]
and we write
$G'(t)$ to denote the continuous operator
\[\begin{split}
G'(t):\Omega^{0,\bullet}_0(U,E)&\To\Omega^{0,\bullet}(U,E),\\
u&\To\int \frac{\pr G(t,z,w)}{\pr t}u(w)dv_U(w).\end{split}\]

We consider the heat operator of $\Box_{B,m}$. By using the standard Dirichlet heat kernel construction (see~\cite{G8}) and the proofs of Theorem 1.6.1 and Theorem 5.5.9 in~\cite{MM}, we deduce the following 

\begin{theorem}\label{t-gue150607}
There is $A_{B,m}(t,z,w)\in C^\infty(\Real_+\times U\times U,(T^{*0,\bullet}_wU\otimes E_w)\boxtimes(T^{*0,\bullet}_zU\otimes E_z))$ such that 
\begin{equation}\label{e-gue150607ab}
\begin{split}
&\mbox{$\lim_{t\To0+}A_{B,m}(t)=I$ in $D'(U,T^{*0,\bullet}U\otimes E)$},\\
&A'_{B,m}(t)u+\frac{1}{m}A_{B,m}(t)(\Box_{B,m}u)=0,\ \ \forall u\in\Omega^{0,\bullet}_0(U,E),\ \ \forall t>0,
\end{split}\end{equation}
and $A_{B,m}(t,z,w)$ satisfies the following: 


(I)  For every compact set $K\Subset U$, $\alpha_1, \alpha_2, \beta_1, \beta_2\in\mathbb N^n_0$,  there are constants $C_{\alpha_1,\alpha_2,\beta_1,\beta_2,K}>0$ and $\varepsilon_0>0$ independent of $t$ and $m$ such that
\begin{equation}\label{e-gue160128w}
\begin{split}
&\abs{\pr^{\alpha_1}_z\pr^{\alpha_2}_{\ol z}\pr^{\beta_1}_w\pr^{\beta_2}_{\ol w}\Bigr(A_{B,m}(t,z,w)e^{m(\varphi(w)-\varphi(z))}\Bigr)}\\
&\leq C_{\alpha_1,\alpha_2,\beta_1,\beta_2,K}(\frac{m}{t})^{n+\abs{\alpha_1}+\abs{\alpha_2}+\abs{\beta_1}+\abs{\beta_2}}e^{-m\varepsilon_0\frac{\abs{z-w}^2}{t}},\ \  \forall (t,z,w)\in \Real_+\times K\times K.
\end{split}
\end{equation}


(II) $A_{B,m}(t,z,z)$ admits an asymptotic expansion: 
\begin{equation}\label{E:1.6.4b}
A_{B,m}(t,z,z)=(2\pi)^{-n} \frac{\det(\dot{R}^L)\exp(t\omega_d)}{\det(1-\exp(-t\dot{R}^L))}(z)  \otimes \operatorname{Id}_{E_z} \otimes m^n+o(m^n)
\end{equation}
in $C^\ell(U,\operatorname{End}(T^{*0,\bullet}U) \otimes E)$ locally uniformly on $\Real_+\times U$, for every $\ell\in\mathbb N$. Here we use the convention that if an eigenvalue $a_j(z)$ of $\dot{R}^L(z)$ is zero, then its contribution for $\frac{\det(\dot{R}^L)}{\det(1-\exp(-t\dot{R}^L))}(z)$ is $\frac{1}{t}$. 

(III) There exist $\mathcal{A}^B_{m,\ell}(z,w)\in C^\infty(U\times U,(T^{*0,\bullet}_wU\otimes E_w)\boxtimes(T^{*0,\bullet}_zU\otimes E_z))$ with 
\begin{equation}\label{e-gue160128bb}
\mathcal{A}^B_{m,\ell}(z,z) =\mathcal{A}_\ell(z) \otimes \operatorname{Id}_E + O(m^{-1/2})
\end{equation}
in $C^0(U, \operatorname{End}(T^{*0,\bullet}U)\otimes E)$ locally uniformly on $U$ and
\begin{equation}\label{e-gue160303}
e^{m(\varphi(w)-\varphi(z))}\mathcal{A}^B_{m,\ell}(z,w)=O(m^0)
\end{equation}
in $C^k(U\times U,(T^{*0,\bullet}U\otimes E)\boxtimes(T^{*0,\bullet}U\otimes E))$ locally uniformly on $U\times U$, $\forall k\in\mathbb N_0$, $\ell=-n,-n+1,\ldots$, where $\mathcal{A}_\ell$ is as in \eqref{e-gue160128aI}, such that for any $k \in \mathbb{N}$ and every compact set $K\Subset U$, there exists $C>0$ independent of $t$ and $m$ such that for any $t \in ]0,1], m \in \mathbb{N}^*$, we have the asymptotic expansion
\begin{equation}\label{e-gue160128y}
\begin{split}
&\Big| m^{-n}e^{m(\varphi(w)-\varphi(z))}A_{B,m}(t,z,w)-e^{-\frac{mh(z,w)}{t}}e^{m(\varphi(w)-\varphi(z))}\sum_{\ell=-n}^k\mathcal{A}^B_{m,\ell}(z,w)t^\ell\Big|\\
&\le Ct^{k},\ \ \forall (z,w)\in K\times K,
\end{split}
\end{equation}
where $h(z,w)\in C^\infty(U\times U)$ and for every compact set $K\Subset U$, there is a constant $C>1$ such that $\frac{1}{C}\abs{z-w}^2\leq h(z,w)\leq C\abs{z-w}^2$, for all $(z,w)\in K\times K$. 
\end{theorem} 

Assume that $X=D_1\bigcup D_2\bigcup\cdots\bigcup D_N$, where $B_j:=(D_j,(z,\theta),\varphi_j)$ is a BRT trivialization, for each $j$. We may assume that for each $j$, $D_j=U_j\times]-2\delta_j,2\Td\delta_j[\subset\Complex^n\times\Real$, $\delta_j>0$, $\Td\delta_j>0$, $U_j=\set{z\in\Complex^n;\, \abs{z}<\gamma_j}$. For each $j$, put $\hat D_j=\hat U_j\times]-\frac{\delta_j}{2},\frac{\Td\delta_j}{2}[$, where $\hat U_j=\set{z\in\Complex^n;\, \abs{z}<\frac{\gamma_j}{2}}$. We may suppose that $X=\hat D_1\bigcup\hat D_2\bigcup\cdots\bigcup\hat D_N$. Let $\chi_j\in C^\infty_0(\hat D_j)$, $j=1,2,\ldots,N$, with $\sum^N_{j=1}\chi_j=1$ on $X$. Fix $j=1,2,\ldots,N$. Put 
\[K_j=\set{z\in\hat U_j;\, \mbox{there is a $\theta\in]-\frac{\delta_j}{2},\frac{\Td\delta_j}{2}[$ such that $\chi_j(z,\theta)\neq0$}}.\]
Let $\tau_j(z)\in C^\infty_0(\hat U_j)$ with $\tau_j\equiv1$ on some neighborhood $W_j$ of $K_j$. Let $\sigma_j\in C^\infty_0(]-\frac{\delta_j}{2},\frac{\Td\delta_j}{2}[)$ with $\int\sigma_j(\theta)d\theta=1$. Let $A_{B_j,m}(t,z,w)\in C^\infty(\Real_+\times U_j\times U_j,(T^{*0,\bullet}_wU_j\otimes E_w)\boxtimes(T^{*0,\bullet}_zU_j\otimes E_z))$ be as in Theorem~\ref{t-gue150607}. Put 
\begin{equation}\label{e-gue150627f}
H_{j,m}(t,x,y)=\chi_j(x)e^{-m\varphi_j(z)+im\theta}A_{B_j,m}(t,z,w)e^{m\varphi_j(w)-im\eta}\tau_j(w)\sigma_j(\eta),
\end{equation}
where $x=(z,\theta)$, $y=(w,\eta)\in\Complex^{n}\times\Real$. Let 
\begin{equation}\label{e-gue150626fIII}
\begin{split}
\Gamma_m(t,x,y):=\frac{1}{2\pi}\sum^N_{j=1}\int^\pi_{-\pi}H_{j,m}(t,x,e^{iu}\circ y)e^{imu}du.
\end{split}
\end{equation}
From Lemma~\ref{l-gue150606}, off-diagonal estimates of $A_{B_j,m}(t,x,y)$ (see \eqref{e-gue160128w} and \eqref{e-gue160128y}), we can repeat the proof of Theorem 5.10 in~\cite{CHT} with minor change and deduce that 

\begin{theorem}\label{t-gue150630I}
For every $\ell\in\mathbb N$, $\ell\geq2$, and every $M>0$, there are $\epsilon_0>0$ and $m_0>0$ independent of $t$ and $m$ such that for every $m\geq m_0$, we have
\begin{equation}\label{e-gue150630g}
\norm{e^{-\frac{t}{m}\Box_{b,m}}(x,y)-\Gamma_m(t,x,y)}_{C^\ell(X\times X)}\leq e^{-\frac{m}{t}\epsilon_0},\ \ \forall t\in(0,M).
\end{equation}
\end{theorem}

\subsection{Proof of Theorem~\ref{T:1.6.1}}
Let $B_j:=(D_j,(z,\theta),\varphi_j)$, $\hat D_j$, $j=1,\ldots,N$ be as before. We will use the same notations as before. For each $j$, we have $D_j=U_j\times]-2\delta_j,2\Td\delta_j[\subset\Complex^n\times\Real$, $\hat D_j=\hat U_j\times]-\frac{\delta_j}{2},\frac{\Td\delta_j}{2}[\subset\Complex^n\times\Real$, $\delta_j>0$, $\Td\delta_j>0$, $U_j=\set{z\in\Complex^n;\, \abs{z}<\gamma_j}$, $\hat U_j=\set{z\in\Complex^n;\, \abs{z}<\frac{\gamma_j}{2}}$. Recall that $X=\hat D_1\bigcup\cdots\bigcup\hat D_N$. We may assume that $\zeta=\delta_j=\Td\delta_j$, $j=1,2,\ldots,N$, and $\zeta>0$ satisfies 
\begin{equation}\label{e-gue160327}
0<\zeta<\inf\set{\frac{\pi}{p_k},\abs{\frac{2\pi}{p_r}-\frac{2\pi}{p_{r+1}}}, r=1,\ldots,k-1}.
\end{equation} 
It is straightforward to see that there is a $\hat\varepsilon_0>0$ such that for each $j=1,\ldots,N$, we have 
\begin{equation}\label{e-gue160215}
\begin{split}
&\hat\varepsilon_0d((z_1,\theta_1),(z_2,\theta_1))\leq\abs{z_1-z_2}\leq\frac{1}{\hat\varepsilon_0}d((z_1,\theta_1),(z_2,\theta_1)), \forall (z_1,\theta_1),(z_2,\theta_1)\in\hat D_j,\\
&\hat\varepsilon_0d((z_1,\theta_1),(z_2,\theta_1))^2\leq h_{j}(z_1,z_2)\leq\frac{1}{\hat\varepsilon_0}d((z_1,\theta_1),(z_2,\theta_1))^2, \forall (z_1,\theta_1),(z_2,\theta_1)\in\hat D_j,
\end{split}
\end{equation}
where $h_j(z,w)$ is as in \eqref{e-gue160128y}. 

Fix $x_0\in X_{{\rm reg\,}}$. Fix $j=1,2,\ldots,N$. Assume that $x_0\in\hat D_j$ and suppose that $x_0=(z_0,0)\in\hat D_j$. Claim that there is a $\Td\varepsilon_0>0$ independent of $x_0$ such that 
\begin{equation}\label{e-gue160327I}
\begin{split}
&\mbox{if $e^{i\theta}\circ x_0=(\Td z,\Td\eta)\in\hat D_j$, for some $\theta\in[\zeta,2\pi-\zeta]$},\\
&\mbox{then $\abs{z_{0}-\Td z}\geq\Td\varepsilon_0d(x_0,X_{{\rm sing\,}})$}. 
\end{split}
\end{equation}

\begin{proof}[Proof of the claim]
From \eqref{e-gue160215}, we have 
\begin{equation}\label{e-gue160327II}
\begin{split}
\abs{\Td z-z_{0}}&\geq\hat\varepsilon_0\inf\set{d(e^{iu}\circ e^{i\theta}\circ x_0,x_0);\, \abs{u}\leq\frac{\zeta}{2}}\\
&\geq\hat\varepsilon_0\inf\set{d(e^{i\theta}\circ x_0,x_0);\, \frac{\zeta}{2}\leq\theta\leq 2\pi-\frac{\zeta}{2}}.
\end{split}
\end{equation}
It is well-known that (see Theorem 6.5 in~\cite{CHT}) there is a constant $C>1$ independent of $x_0$ such that 
\begin{equation}\label{e-gue160423}
\frac{1}{C}d(x_0,X_{{\rm sing\,}})\leq\inf\set{d(e^{i\theta}\circ x_0,x_0);\, \frac{\zeta}{2}\leq\theta\leq 2\pi-\frac{\zeta}{2}}\leq Cd(x_0,X_{{\rm sing\,}}).
\end{equation}
From \eqref{e-gue160423} and \eqref{e-gue160327II}, the claim follows.
\end{proof}

From \eqref{e-gue160215}, \eqref{e-gue160327I} and \eqref{e-gue160128y}, there are $\varepsilon_0>0$ and $C_0>0$ independent of $j$, $x_0$, $m$, $t$ such that
\begin{equation}\label{e-gue160125V}
\abs{\frac{1}{2\pi}\int_{u\in[\zeta,2\pi-\zeta]}H_{j,m}(t,x_0,e^{iu}\circ x_0)e^{imu}du}\leq C_0m^nt^{-n}e^{\frac{-\varepsilon_0m d(x_0,X_{{\rm sing\,}})^2}{t}},\ \  \forall t\in\Real_+,\ \ \forall m\in\mathbb N.
\end{equation}
Moreover, it is obvious that 
\begin{equation}\label{e-gue160125VI}
\frac{1}{2\pi}\int^{\zeta}_{-\zeta}H_{j,m}(t,x_0,e^{iu}\circ x_0)e^{imu}du=\frac{1}{2\pi}\chi_j(x_0)A_{B_j,m}(t,z_0,z_0).
\end{equation}
From \eqref{e-gue160125V}, \eqref{e-gue160125VI} and \eqref{E:1.6.4b}, we conclude that for every $\varepsilon>0$ and every $I\Subset\Real_+$, there are $m_0>0$, $\varepsilon_0>0$, $C_0>0$ independent of $x_0$, $t$, $m$ such that 
\begin{equation}\label{e-gue160304}
\begin{split}
&\abs{\frac{1}{2\pi}\int^{\pi}_{-\pi}H_{j,m}(t,x_0,e^{iu}\circ x_0)e^{imu}du-(2\pi)^{-n-1}\chi_j(x_0) \frac{\det(\dot{R}^L)\exp(t\omega_d)}{\det(1-\exp(-t\dot{R}^L))}(z_0)  \otimes \operatorname{Id}_{E_{z_0}} \otimes m^n}\\
&\leq \varepsilon m^n+C_0m^ne^{-\varepsilon_0md(x_0,X_{{\rm sing\,}})^2},
\end{split}
\end{equation}
for every $m\geq m_0$ and every $t\in I$. 

From \eqref{e-gue160304}, \eqref{e-gue150626fIII} and \eqref{e-gue160129I}, we deduce that for every $\varepsilon>0$ and every $I\Subset\Real_+$, there are constants $m_0>0$, $\varepsilon_0>0$ and $C>0$ independent of $x_0$, $m$ such that 
\begin{equation}\label{e-gue160215VII}
\abs{\Gamma(t,x_0,x_0)-(2\pi)^{-n-1}\frac{\det(\dot{\mathcal{R}}) \exp(t \gamma_d)}{\det(1-\exp(-t\dot{\mathcal{R}}))}(x_0)  \otimes \operatorname{Id}_{E_{x_0}} \otimes m^n}\leq \varepsilon m^n+Cm^ne^{-\varepsilon_0md(x_0,X_{{\rm sing\,}})^2},
\end{equation}
for all $t\in I$, $m\geq m_0$. 

From \eqref{e-gue160215VII}, \eqref{e-gue150630g} and notice that the constants $m_0$, $C$ and $\varepsilon_0$ are independent of $x_0$, we get \eqref{E:1.6.4}.

\subsection{Asymptotic expansion for the heat kernels of the Kohn Laplacians}

Now, we can prove 

\begin{theorem}\label{t-gue160303}
There exist $A_{m,\ell}(t,x)\in C^\infty(\Real_+\times X, \operatorname{End}(T^{*0,\bullet}X\otimes E))$ with $\abs{A_{m,\ell}(t,x)}\leq C_\ell$, for every $(t,x)\in\Real_+\times X$, where $C_\ell>0$ is a constant independent of $m$, $\ell=-n,-n+1,\ldots$, such that for any $k \in \mathbb{N}$, there exists $C>0$ such that for any $t \in ]0,1], m \in \mathbb{N}^*$ and every $x\in X$, we have the asymptotic expansion
\begin{equation}\label{e-gue160130q}
\Big| m^{-n}e^{-\frac{t}{m} \Box_{b,m}}(x,x) - \sum_{\ell=-n}^k A_{m,\ell}(t,x)t^j \Big| \le Ct^k,
\end{equation}  
Moreover, for every $\ell=-n,-n+1,\ldots$, we can find $A_{m,\ell}(x)\in C^\infty(X, \operatorname{End}(T^{*0,\bullet}_xX\otimes E_x))$ with
\begin{equation}\label{e-gue160128b}
A_{m,\ell}(x) = A_\ell(x) \otimes \operatorname{Id}_{E_x} + O(m^{-1/2})
\end{equation}
in $C^0(X,\operatorname{End}(T^{*0,\bullet}X)\otimes E)$ uniformly on $X$ such that there exist $C>0$ and $\varepsilon_0>0$ independent of $t$ and $m$ such that for any $t \in ]0,1], m \in \mathbb{N}^*$, we have
\begin{equation}\label{e-gue160130qm}
\Big|A_{m,\ell}(t,x)-A_{m,\ell}(x)\Big| \le Ce^{-\frac{\varepsilon_0md(x,X_{{\rm sing\,}})^2}{t}},\ \ \forall x\in X_{\rm reg},
\end{equation}
where $A_\ell(x)\in C^\infty(X, \operatorname{End}(T^{*0,\bullet}X)\otimes E)$, $\ell=-n,-n+1,\ldots$, are as in \eqref{E:5.5.34}. 
\end{theorem}

\begin{proof}
Let $B_j:=(D_j,(z,\theta),\varphi_j)$, $\hat D_j$, $j=1,\ldots,N$ be as before. We will use the same notations as before.
Let $\hat\sigma_j\in C^\infty_0(]-\delta_j,\Td\delta_j[)$ such that $\hat\sigma_j=1$ on some neighbourhood of ${\rm Supp\,}\sigma_j$ and $\hat\sigma_j(\theta)=1$ if $(z,\theta)\in{\rm Supp\,}\chi_j$. For each $j=1,2,\ldots,N$, put 
\begin{equation}\label{e-gue150901}
\begin{split}
&\hat h_{j}(x,y)=\hat\sigma_j(\theta) h_{j}(z,w)\hat\sigma_j(\eta)\in C^\infty_0(D_j),\ \ x=(z,\theta),\ \ y=(w,\eta),\\
&\hat{\mathcal{A}}^{B_j}_{m,\ell}(x,y)=\chi_j(x)e^{-m\varphi_j(z)+im\theta}\mathcal{A}^{B_j}_{m,\ell}(z,w)e^{m\varphi_j(w)-im\eta}\tau_j(w)\sigma_j(\eta),\ \ \ell=-n,-n+1,\ldots,
\end{split}
\end{equation}
where $\mathcal{A}^{B_j}_{m,\ell}$ is as in \eqref{e-gue160128bb}. 
Let 
\begin{equation}\label{e-gue160304p}
\begin{split}
&A_{m,\ell}(t,x,y)=\frac{1}{2\pi}\sum^N_{j=1}\int^\pi_{-\pi}e^{-m\frac{\hat h_{j}(x,e^{iu}\circ y)}{t}}\hat{\mathcal{A}}^{B_j}_{m,\ell}(x,e^{iu}\circ y)e^{imu}du,\ \ \ell=-n,-n+1,-n+2,\ldots,\\ 
&A_{m,\ell}(t,x):=A_{m,\ell}(t,x,x),\ \ \ell=-n,-n+1,-n+2,\ldots.
\end{split}
\end{equation}
From Theorem~\ref{t-gue150630I} and \eqref{e-gue160128y}, it is easy to see that \eqref{e-gue160130q} holds.

Put 
\begin{equation}\label{e-gue160304pI}
A_{m,\ell}(x):=\frac{1}{2\pi}\sum^N_{j=1}\chi_j(x)\mathcal{A}^{B_j}_{m,\ell}(z,z),\ \ \ell=-n,-n+1,-n+2,\ldots.
\end{equation}
From \eqref{e-gue160128bb} and \eqref{e-gue160131}, we have 
\begin{equation}\label{e-gue160304pII}
 A_{m,\ell}(x) = A_\ell(x) \otimes \operatorname{Id}_{E_x} + O(m^{-1/2})
\end{equation}
in $C^0(X,\operatorname{End}(T^{*0,\bullet}_xX\otimes E_x))$ uniformly on $X$,
where $A_\ell(x)\in C^\infty(X, \operatorname{End}(T^{*0,\bullet}_xX\otimes E_x))$, $\ell=-n,-n+1,\ldots$, are as in \eqref{E:5.5.34}. Fix $\ell=-n,-n+1,\ldots$ and fix $x_0\in X_{{\rm reg\,}}$. As in the proof of Theorem~\ref{T:1.6.1},, we have  
\begin{equation}\label{e-gue160226II}
A_{m,\ell}(t,x_0)=A_{m,\ell}(x_0)+\frac{1}{2\pi}\sum^N_{j=1}\int_{u\in[\zeta,2\pi-\zeta]}e^{-m\frac{\hat h_{j}(x_0,e^{iu}\circ x_0)}{t}}\hat{\mathcal{A}}^{B_j}_{m,\ell}(x_0,e^{iu}\circ x_0)e^{imu}du.
\end{equation}
From the property of $h_j(z,w)$, we can repeat the proof of Theorem~\ref{T:1.6.1} and conclude that for every $j=1,2,\ldots,N$, there exist $C>0$ and $\varepsilon_0>0$ independent of $t$, $m$ and $x_0$ such that for any $t \in ]0,1], m \in \mathbb{N}^*$ and any $u\in[\zeta,2\pi-\zeta]$ with $e^{iu}\circ x_0\in\hat D_j$ , we have
\begin{equation}\label{e-gue160423q}
\abs{e^{-m\frac{\hat h_{j}(x_0,e^{iu}\circ x_0)}{t}}}\le Ce^{-\frac{\varepsilon_0md(x_0,X_{{\rm sing\,}})^2}{t}}.
\end{equation}
From \eqref{e-gue160423q} and \eqref{e-gue160226II}, the theorem follows. 
\end{proof}

We pause and introduce some notations. Let $g(t)\in C^\infty(\Real_+)$. We write 
\[g(t)\sim \sum^{\infty}_{j=0}g_{-k+\frac{j}{2}}t^{-k+\frac{j}{2}}\ \ \mbox{as $t\To0^+$},\]
where $k\in\mathbb N_0$, $g_{-k+\frac{j}{2}}\in\mathbb C$, $j=0,1,2,\ldots$, if for any $N\in\mathbb N$, there exist $C_N>0$ and $\delta>0$ independent of $t$ such that 
\begin{equation}\label{e-gue160420a}
\abs{g(t)-\sum^N_{j=0}g_{-k+\frac{j}{2}}t^{-k+\frac{j}{2}}}\leq C_Nt^{-k+\frac{N+1}{2}},\ \ \forall t\in]0,\delta[.
\end{equation}
Let $h_m(t)\in C^\infty(\Real_+)$ be $m$-dependent function. We write 
\[h_m(t)\sim \sum^{\infty}_{j=0}h_{m,-k+\frac{j}{2}}t^{-k+\frac{j}{2}}\ \ \mbox{as $t\To0^+$, uniformly in $m$},\]
where $k\in\mathbb N_0$, $h_{m,-k+\frac{j}{2}}\in\mathbb C$, $j=0,1,2,\ldots$, if for any $N\in\mathbb N$, there exist $C_N>0$, $\delta>0$ independent of $t$ and $m$ such that for all $m\in\mathbb N$, 
\begin{equation}\label{e-gue160420aw}
\abs{h_m(t)-\sum^N_{j=0}h_{m,-k+\frac{j}{2}}t^{-k+\frac{j}{2}}}\leq C_Nt^{-k+\frac{N+1}{2}},\ \ \forall t\in]0,\delta[.
\end{equation}

For the proof of our main result, we need to know the asymptotic behavior of 
\[\int_X A_{m,\ell}(t,x)dv_X(x),\ \ \ell=-n,-n+1,\ldots.\]
Our next goal in this section is to prove the following

\begin{theorem}\label{t-gue160423}
With the notations used in Theorem~\ref{t-gue160303}, fix $\ell=-n,-n+1,\ldots$. We can find $a^{\frac{j}{2}}_{m,\ell}\in\Real$, $j=0,1,2,\ldots$, with 
\begin{equation}\label{e-gue160425b}
\begin{split}
&\lim_{m\To\infty}a^0_{m,\ell}=\int_XA_{\ell}(x)\otimes\operatorname{Id}_{E_x}dv_X(x),\\
&\lim_{m\To\infty}a^{\frac{j}{2}}_{m,\ell}=0,\ \ j=1,2,\ldots,
\end{split}
\end{equation}
$a^{\frac{j}{2}}_{m,\ell}$ is independent of $t$, for each $j$, $\abs{a^{\frac{j}{2}}_{m,\ell}}\leq C_j$, for every $m\in\mathbb N$, where $C_j>0$ is a constant independent of $m$,  $j=0,1,2,\ldots$, such that 
\begin{equation}\label{e-gue160416I}
\int_X A_{m,\ell}(t,x)dv_X(x)\sim\sum^{\infty}_{j=0}a^{\frac{j}{2}}_{m,\ell}t^{\frac{j}{2}}\ \ \ \mbox{as $t\To0^+$, uniformly in $m$}.
\end{equation}

In particular, 
\begin{equation}\label{e-gue160416II}
\begin{split}
&\int_Xm^{-n}e^{-\frac{t}{m}\Box_{b,m}}(x,x)dv_X(x)\\
&\sim t^{-n}b_{m,-n}+t^{-n+\frac{1}{2}}b_{m,-n+\frac{1}{2}}+t^{-n+1}b_{-n+1}+t^{-n+\frac{3}{2}}b_{-n+\frac{3}{2}}+\cdots\ \ \mbox{as $t\To0^+$, uniformly in $m$},
\end{split}
\end{equation}
where for each $j=0,\frac{1}{2},1,\ldots$, $b_{m,-n+j}\in\Complex$ is independent of $t$ and there is a constant $\hat C_j>0$ independent of $m$, such that $\abs{b_{m,-n+j}}\leq\hat C_j$, for every $m\in\mathbb N$, and 
\begin{equation}\label{e-gue160426}
\begin{split}
&\lim_{m\To\infty}b_{m,-n+\frac{j}{2}}=\int_XA_{-n+\frac{j}{2}}(x)\otimes\operatorname{Id}_{E_x}dv_X(x) \ \mbox{if $j$ is an even number},\\
&\lim_{m\To\infty}b_{m,-n+\frac{j}{2}}=0,\ \ \mbox{if $j$ is an odd number}.
\end{split}
\end{equation}
\end{theorem}
Let $B_j:=(D_j,(z,\theta),\varphi_j)$, $\hat D_j$, $j=1,\ldots,N$ be as before. We will use the same notations as before. Let $\hat h_{j}(x,y)$, $\hat{\mathcal{A}}^{B_j}_{m,\ell}(x,y)$ be as in \eqref{e-gue150901}, $\ell=-n,-n+1,\ldots$. Now, we fix $\ell=-n,-n+1,\ldots$ and $j=1,2,\ldots,N$. 

\begin{theorem}\label{t-gue160416b}
Fix $p_s$, $s=2,3,\ldots,k$, and $x_0\in\hat D_j$. Assume that $e^{i\frac{2\pi}{p_s}}\circ x_0\in\hat D_j$. Then, there are a neighborhood $\Omega$ of $x_0$ and $\varepsilon>0$ such that for every $g(x)\in C^\infty_0(\Omega)$, we have 
\begin{equation}\label{e-gue160416f}
\begin{split}
&\int^{\frac{2\pi}{p_s}+\varepsilon}_{\frac{2\pi}{p_s}-\varepsilon}\int_Xg(x)e^{-m\frac{\hat h_{j}(x,e^{iu}\circ x)}{t}}\hat{\mathcal{A}}^{B_j}_{m,\ell}(x,e^{iu}\circ x)e^{imu}dv_X(x)du\\
&\sim b_{m,0}+b_{m,\frac{1}{2}}t^{\frac{1}{2}}+b_{m,1}t+b_{m,\frac{3}{2}}t^{\frac{3}{2}}+\cdots\mbox{as $t\To0^+$, uniformly in $m$}, 
\end{split}
\end{equation}
where $b_{m,j}\in\Complex$ is independent of $t$ and there is a constant $\hat C_j>0$ independent of $m$, such that $\abs{b_{m,j}}\leq\hat C_j$, for every $m\in\mathbb N$, and $\lim_{m\To\infty}b_{m,j}=0$, $j=0,\frac{1}{2},1,\frac{3}{2},\ldots$.
\end{theorem}

\begin{proof}
For simplicity, assume that $x_0=(z_0,0)$. Suppose $e^{i\frac{2\pi}{p_s}}\circ x_0=(\Td z_0,\Td\theta)$. If $\Td z_0\neq z_0$, then there are a neighborhood $\Omega$ of $x_0$ and $\varepsilon>0$ such that $e^{i\theta}\circ x=(\Td z,\hat\theta)$ with $\abs{\Td z-z_0}\geq\frac{1}{2}\abs{z_0-\Td z_0}$, for every $x\in\Omega$ and every $\theta\in]\frac{2\pi}{p_s}-\varepsilon,\frac{2\pi}{p_s}+\varepsilon[$. From  the property of $\hat h_{j}(x,y)$ (see \eqref{e-gue160128y} and \eqref{e-gue150901}), it is straightforward to see that
\begin{equation}\label{e-gue160417J}
\int^{\frac{2\pi}{p_s}+\varepsilon}_{\frac{2\pi}{p_s}-\varepsilon}\int_Xg(x)e^{-m\frac{\hat h_{j}(x,e^{iu}\circ x)}{t}}\hat {\mathcal{A}}^{B_j}_{m,\ell}(x,e^{iu}\circ x)e^{imu}dv_X(x)du=O((\abs{\frac{t}{m}})^\infty),\end{equation}
for every $g\in C^\infty_0(\Omega)$. Thus, we may assume that $e^{i\frac{2\pi}{p_s}}\circ x_0=(z_0,\Td\theta_0)$, $\abs{\Td\theta_0}<\frac{\zeta}{2}$. It $\Td\theta_0\neq0$, then $e^{i\frac{2\pi}{p_\ell}-i\Td\theta_0} x_0=x_0$. Hence, 
$\frac{2\pi}{p_s}-\Td\theta_0=\frac{2\pi}{p_k}$, for some $p_k$. But $\abs{\Td\theta_0}<\frac{\zeta}{2}$ and $\zeta$ satisfies \eqref{e-gue160327}, we get a contradiction. Hence, $\Td\theta_0=0$. Thus, 
 $e^{i\frac{2\pi}{p_s}}\circ x_0=(z_0,0)=x_0$. 
 
For every $x$ in some small neighborhood of $x_0$ and every $\theta\in]\frac{2\pi}{p_s}-\varepsilon,\frac{2\pi}{p_s}+\varepsilon[$, where $0<\varepsilon<\frac{\zeta}{2}$ is a small constant, if $x=(z,v)\in\hat D_j$, $e^{i\theta}\circ x=(\Td z,\Td v)\in\hat D_j$, we can check that
 \begin{equation}\label{e-gue160417a}
h_{j}(z,\Td z)=f(x)\abs{z-\Td z}^2=f(x)\inf\set{\abs{x-e^{iu+i\theta} x}^2;\, \abs{u}\leq\frac{\zeta}{2}},
 \end{equation}
where $h_{j}$ is as in \eqref{e-gue160128y}, $f(x)$ is a positive continuous function. From \eqref{e-gue160417a}, we can repeat the proof of Theorem 6.5 in~\cite{CHT} with minor change and deduce that for every $x\in X_{{\rm reg\,}}$ and $x$ in some small neighborhood of $x_0$ and every $\theta\in]\frac{2\pi}{p_s}-\varepsilon,\frac{2\pi}{p_s}+\varepsilon[$, where $0<\varepsilon<\frac{\zeta}{2}$ is a small constant, we have
 \begin{equation}\label{e-gue160417aI}
h_{j}(z,\Td z)=f_1(x)d(x,X_{p_\ell})^2,
\end{equation}
where $x=(z,v)\in\hat D_j$, $e^{i\theta}x=(\Td z,\Td v)\in\hat D_j$ and $f_1(x)$ is a positive continuous function. We take local coordinates $y=(y_1,\ldots,y_{2n+1})$ defined in some small neighborhood $\hat\Omega\Subset\hat D_j$ of $x_0$ such that $y(x_0)=0$, $X_{p_s}=\set{y\in\hat\Omega;\, y_1=\cdots=y_r=0}$ and 
\begin{equation}\label{e-gue160417aIII}
d(y,X_{p_s})^2=f_2(y)(\abs{y_1}^2+\cdots+\abs{y_r}^2),\ \ \forall y\in\hat\Omega,
\end{equation}
where $f_2(y)$ is a positive continuous function. Let $\Omega\Subset\hat\Omega$ be an open set of $x_0$ and let $0<\varepsilon<\frac{\zeta}{2}$ be a small constant so that $e^{i\theta}\circ x\in\hat\Omega$, for every $x\in\Omega$ and every $\theta\in]\frac{2\pi}{p_s}-\varepsilon,\frac{2\pi}{p_s}+\varepsilon[$. From \eqref{e-gue160417aI} and \eqref{e-gue160417aIII}, for every $g\in C^\infty_0(\Omega)$, we have 
\begin{equation}\label{e-gue160417aII}
\begin{split}
&\int^{\frac{2\pi}{p_s}+\varepsilon}_{\frac{2\pi}{p_s}-\varepsilon}\int_Xg(x)e^{-m\frac{\hat h_{j}(x,e^{iu}\circ x)}{t}}\hat{\mathcal{A}}^{B_j}_{m,\ell}(x,e^{iu}\circ x)e^{imu}dv_X(x)du\\
&=\int^{\frac{2\pi}{p_s}+\varepsilon}_{\frac{2\pi}{p_s}-\varepsilon}\int_Xg(x)e^{-m\frac{f_1(x)d(x,X_{p_s})^2}{t}}\hat{\mathcal{A}}^{B_j}_{m,\ell}(x,e^{iu}\circ x)e^{imu}dv_X(x)du\\
&=\int^{\frac{2\pi}{p_s}+\varepsilon}_{\frac{2\pi}{p_s}-\varepsilon}\int_Xg(y)e^{-m\frac{f_1(y)f_2(y)(\abs{y_1}^2+\cdots+\abs{y_r})^2}{t}}\hat{\mathcal{A}}^{B_j}_{m,\ell}(y,e^{iu}\circ y)e^{imu}dv_X(y)du\\
&\sim b_{m,0}+b_{m,\frac{1}{2}}t^{\frac{1}{2}}+b_1t+b_{m,\frac{3}{2}}t^{\frac{3}{2}}+\cdots\mbox{as $t\To0^+$, uniformly in $m$}, 
\end{split}
\end{equation}
where $b_{m,j}\in\Complex$ is independent of $t$ and there is a constant $\hat C_j>0$ independent of $m$, such that $\abs{b_{m,j}}\leq\hat C_j$, for every $m\in\mathbb N$, and $\lim_{m\To\infty}b_{m,j}=0$, $j=0,\frac{1}{2},1,\frac{3}{2},\ldots$.
From \eqref{e-gue160417J} and \eqref{e-gue160417aII}, the theorem follows. 
\end{proof}

\begin{theorem}\label{t-gue160417}
Let $x_0\in\hat D_j$. Then, there is a neighborhood $\Omega$ of $x_0$ such that for every $g\in C^\infty_0(\Omega)$, we have
\begin{equation}\label{e-gue160417b}
\begin{split}
&\int^{2\pi}_0\int_Xg(x)e^{-m\frac{\hat h_{j}(x,e^{iu}\circ x)}{t}}\hat{\mathcal{A}}^{B_j}_{m,\ell}(x,e^{iu}\circ x)e^{imu}dv_X(x)du\\
&\sim d_{m,0}+d_{m,\frac{1}{2}}t^{\frac{1}{2}}+d_{m,1}t+d_{m,\frac{3}{2}}t^{\frac{3}{2}}+\cdots\mbox{as $t\To0^+$, uniformly in $m$}, 
\end{split}
\end{equation}
where $d_{m,j}\in\Complex$ is independent of $t$ and there is a constant $\hat C_j>0$ independent of $m$, such that $\abs{d_{m,j}}\leq\hat C_j$, for every $m\in\mathbb N$, and 
\begin{equation}\label{e-gue160426y}
\begin{split}
&\lim_{m\To\infty}d_{m,0}=\lim_{m\To\infty}\int_Xg(x)\hat{\mathcal{A}}^{B_j}_{m,\ell}(x,x)dv_X(x)du,\\
&\lim_{m\To\infty}d_{m,j}=0,\ \ j=\frac{1}{2},1,\frac{3}{2},\ldots.
\end{split}
\end{equation}
\end{theorem}

\begin{proof}
From Theorem~\ref{t-gue160416b}, we see that there are a neighborhood $\Omega$ of $x_0$ and $0<\varepsilon<\frac{\zeta}{2}$ such that for every $g\in C^\infty_0(\Omega)$ and every $s=2,3,\ldots,k$, we have 
\begin{equation}\label{e-gue160416fq}
\begin{split}
&\int^{\frac{2\pi}{p_s}+\varepsilon}_{\frac{2\pi}{p_s}-\varepsilon}\int_Xg(x)e^{-m\frac{\hat h_{j}(x,e^{iu}\circ x)}{t}}\hat{\mathcal{A}}^{B_j}_{m,\ell}(x,e^{iu}\circ x)e^{imu}dv_X(x)du\\
&\sim b_{m,0,s}+b_{m,\frac{1}{2},s}t^{\frac{1}{2}}+b_{m,1,s}t+b_{m,s,\frac{3}{2},\ell}t^{\frac{3}{2}}+\cdots\mbox{as $t\To0^+$, uniformly in $m$}, 
\end{split}
\end{equation}
where $b_{m,j,s}\in\Complex$ is independent of $t$ and there is a constant $\hat C_j>0$ independent of $m$, such that $\abs{b_{m,j,s}}\leq\hat C_j$, for every $m\in\mathbb N$, and $\lim_{m\To\infty}b_{m,j,s}=0$, $j=0,\frac{1}{2},1,\frac{3}{2},\ldots$.

It is obvious that for every $g\in C^\infty_0(\Omega)$,
\begin{equation}\label{e-gue160416fI}
\begin{split}
&\int^{\frac{\zeta}{2}}_{-\frac{\zeta}{2}}\int_Xg(x)e^{-\frac{\hat h_{j}(x,e^{iu}\circ x)}{t}}\hat{\mathcal{A}}^{B_j}_{m,\ell}(x,e^{iu}\circ x)e^{imu}dv_X(x)du=\int g(x)t\hat{\mathcal{A}}^{B_j}_{m,\ell}(x,x)dv_X(x).\end{split}
\end{equation}
Note that when $\abs{u}\leq\frac{\zeta}{2}$, the function $e^{-\frac{\hat h_{j}(x,e^{iu}\circ x)}{t}}$ is independent of $t$, for $x\in\hat D_j$. Moreover, it is not difficult to see that for every $g\in C^\infty_0(\Omega)$, 
\begin{equation}\label{e-gue160416fII}
\int_{\theta\in[\frac{\zeta}{2},2\pi-\frac{\zeta}{2}], \theta\notin\bigcup^k_{\ell=2}[\frac{2\pi}{p_\ell}-\varepsilon,\frac{2\pi}{p_\ell}+\varepsilon]}\int_Xg(x)e^{-\frac{\hat h_{j}(x,e^{iu}\circ x)}{t}}\hat{\mathcal{A}}^{B_j}_{m,\ell}(x,e^{iu}\circ x)e^{imu}dv_X(x)du=O((\abs{\frac{t}{m}})^\infty).
\end{equation}
From \eqref{e-gue160416fq}, \eqref{e-gue160416fI} and \eqref{e-gue160416fII}, the theorem follows. 
\end{proof}

From Theorem~\ref{t-gue160417} and by using partition of unity, we get Theorem~\ref{t-gue160423}.

\subsection{Spectral gap of $\Box^{(q)}_{b,m}$}

Fix $q=0,1,\ldots,n$. Let $\Box^{(q)}_{b,m}:{\rm Dom\,}\Box^{(q)}_{b,m}\subset L^2_m(X,T^{*0,q}X\otimes E)\To L^2_m(X,T^{*0,q}X\otimes E)$ be the restriction of $\Box_{b,m}$ on $(0,q)$ forms. In this work, we need 

\begin{theorem}\label{T:1.5.5}
Let $\mu^{(q)}_m$ be the lowest eigenvalue of $\Box^{(q)}_{b,m}$. There exist constants $c_1>0, c_2>0$ not depending on $m$ such that for $q \ge 1$ and $m\in\mathbb N$, 
\begin{equation}\label{E:1.5.23}
\mu^{(q)}_m \ge c_1m-c_2.
\end{equation}
\end{theorem}

\begin{proof}
Fix $q\in\set{1,2,\ldots,n}$. 
If we go through Kohn's $L^2$ estimates (see~\cite[Theorem 8.4.2]{CS01}),
we see that there is a constant $C>0$ independent of $m$ such that for all
$u\in\Omega^{0,q}_m(X,E)$ with $i(\,Tu\,|\,u\,)_E\leq0$, we have
\begin{equation}\label{e-gue160131wa}
\norm{u}_{E,1}\leq C\Bigr(\norm{\Box^{(q)}_{b,m}u}_E+\norm{u}_E\Bigr),
\end{equation}
where $\norm{\cdot}_{E,1}$ denotes the standard Sobolev norm of order $1$. Now, we assume that $m\in\mathbb N$. Then, $i(\,Tu\,|\,u\,)_E=-m\norm{u}^2_E\leq0$, for all $u\in\Omega^{0,q}_m(X,E)$. From this observation and \eqref{e-gue160131wa}, we get 
\begin{equation}\label{e-gue160131waI}
m\norm{u}_E=\norm{Tu}_E\leq \norm{u}_{E,1}\leq C\Bigr(\norm{\Box^{(q)}_{b,m}u}_E+\norm{u}_E\Bigr), 
\end{equation}
for all  $u\in\Omega^{0,q}_m(X,E)$. Let $\mu^{(q)}_m$ be the lowest eigenvalue of $\Box^{(q)}_{b,m}$ and let $v\in\Omega^{0,q}_m(X,E)$ be a non-zero eigenfunction of $\Box^{(q)}_{b,m}$ with eigenvalue $\mu^{(q)}_m$. From \eqref{e-gue160131waI}, we have 
\begin{equation}\label{e-gue160131waII}
\begin{split}
m\norm{v}_E\leq C\Bigr(\norm{\Box^{(q)}_{b,m}v}_E+\norm{v}_E\Bigr)=C(\mu^{(q)}_m+1)\norm{v}_E.
\end{split}
\end{equation}
Hence, $\mu^{(q)}_m\geq\frac{m}{C}-1$. The theorem follows. 
\end{proof}


\section{Analytic torsion on CR manifolds with $S^1$-action}

In this section we first study Mellin transformation, then we define the Fourier components of the analytic torsion for the rigid CR vector bundle $E$ over the CR manifold $X$ with transversal CR $S^1$-action.

\subsection{Mellin transformation} 

Let $\Gamma(z)$ be the Gamma function on $\Complex$. Then for ${\rm Re\,}z>0$, we have 
\[\Gamma(z)=\int^\infty_0e^{-t}t^{z-1}dt.\]
$\Gamma(z)^{-1}$ is an entire function on $\Complex$ and 
\begin{equation}\label{e-gue160313b}
\Gamma(z)^{-1}=z+O(z^2)\ \ \mbox{near $z=0$}. 
\end{equation}

We suppose that $f(t)\in C^\infty(\Real_+)$ verifies the following two conditions:
\begin{itemize}
\item[I.] 
\begin{equation}\label{e-gue160420g}
\mbox{$f(t)\sim \sum^{\infty}_{j=0}f_{-k+\frac{j}{2}}t^{-k+\frac{j}{2}}$ as $t\To0^+$}, 
\end{equation}
where $k\in\mathbb N_0$, $f_{-k+\frac{j}{2}}\in\mathbb C$, $j=0,1,2,\ldots$.
\item[II.]  For every $\delta>0$, there exist $c>0$, $C>0$ such that 
\begin{equation}\label{e-gue160420I}
\abs{f(t)}\leq Ce^{-ct},\ \ \forall t\geq\delta. 
\end{equation}
\end{itemize}

\begin{definition}\label{d-gue160313}
The \emph{Mellin transformation} of $f$ is the function defined for ${\rm Re\,}z>k$, 
\begin{equation}\label{e-gue160313bIII}
M[f](z)=\frac{1}{\Gamma(z)}\int^\infty_0 f(t)t^{z-1}dt.
\end{equation}
\end{definition}

We can repeat the proof of Lemma 5.5.2 in \cite{MM} and deduce the following 

\begin{theorem}\label{t-gue160313}
$M[f]$ extends to a meromorphic function on $\Complex$ with poles contained in 
\[\set{\ell-\frac{j}{2};\, \ell,j\in\mathbb Z},\] 
and its possible poles are simple. Moreover, $M[f]$ is holomorphic at $0$, 
\begin{equation}\label{e-gue160428w}
M[f](0)=f_0
\end{equation}
and
\begin{equation}\label{e-gue160421s}
\begin{split}
&\frac{\pr M[f]}{\pr z}(0)=\int^1_0(f(t)-\sum^{2k}_{j=0}f_{-k+\frac{j}{2}}t^{-k+\frac{j}{2}})\frac{1}{t}dt\\
&\quad+\int^\infty_1f(t)\frac{1}{t}dt+\sum^{2k-1}_{j=0}\frac{f_{-k+\frac{j}{2}}}{\frac{j}{2}-k}-\Gamma'(1)f_{0}.
\end{split}
\end{equation} 
\end{theorem}

\begin{proof}
By \eqref{e-gue160420I}, the function $\int^\infty_1f(t)t^{z-1}dt$ is an entire function on $z\in\Complex$. For any $N\in\mathbb N$, we have 
\begin{equation}\label{e-gue160421}
\begin{split}
&\int^1_0f(t)t^{z-1}dt\\
&=\int^1_0(f(t)-\sum^N_{j=0}f_{-k+\frac{j}{2}}t^{-k+\frac{j}{2}})t^{z-1}dt+\sum^N_{j=0}f_{-k+\frac{j}{2}}\int^1_0t^{-k+\frac{j}{2}+z-1}dt\\
&=\int^1_0(f(t)-\sum^N_{j=0}f_{-k+\frac{j}{2}}t^{-k+\frac{j}{2}})t^{z-1}dt+\sum^N_{j=0}f_{-k+\frac{j}{2}}\frac{1}{-k+\frac{j}{2}+z}. 
\end{split}
\end{equation}
From \eqref{e-gue160420g} and \eqref{e-gue160420I}, we see that $\int^1_0(f(t)-\sum^N_{j=0}f_{-k+\frac{j}{2}}t^{-k+\frac{j}{2}})t^{z-1}dt$ is a holomorphic function for ${\rm Re\,}z>k-\frac{N+1}{2}+1$. From this observation and \eqref{e-gue160421}, we conclude that $M[f]$ can be extended to a meromorphic function on $\Complex$ with poles contained in 
$\set{\ell-\frac{j}{2};\, \ell,j\in\mathbb Z}$, and its possible poles are simple.

From \eqref{e-gue160313b}, we conclude that $M[f]$ is holomorphic at $z=0$. Take $N=2k+2$ in \eqref{e-gue160421} and by some direct computation, we get \eqref{e-gue160428w} and \eqref{e-gue160421s}.
\end{proof}

\subsection{Definition of the Fourier components of the analytic torsion}\label{s-gue160502q}

Let $N$ be the number operator on $T^{*0,\bullet}X$, i.e. $N$ acts on $T^{*0,q}X$ by multiplication by $q$.
Fix $q=0, 1, \cdots, n$ and take a point $x \in X$. Let $e_1(x), \cdots, e_d(x)$ be an orthonormal frame of $T_x^{*0,q}X \otimes E_x$. Let $A \in (T_x^{*0,\bullet}X \otimes E_x)\boxtimes (T_x^{*0,\bullet}X \otimes E_x)$. Put $\operatorname{Tr}^{(q)} A := \sum_{j=1}^d \langle Ae_j | e_j \rangle_E$ and set
\begin{eqnarray}
\operatorname{Tr}A:= \sum_{j=0}^n \operatorname{Tr}^{(j)}A, \nonumber \\
\operatorname{STr}A:= \sum_{j=0}^n (-1)^j \operatorname{Tr}^{(j)}A.
\end{eqnarray}
Let $A:C^\infty(X,T^{*0,\bullet}X\otimes E)\To C^\infty(X,T^{*0,\bullet}X\otimes E)$ be a continuous operator with distribution kernel $A(x,y)\in C^\infty(X\times X,(T^{*0,\bullet}_yX\otimes E_y)\boxtimes(T^{*0,\bullet}_xX\otimes E_x))$. 
We set 
\[\operatorname{Tr}^{(q)}\lbrack A\rbrack:=\int_X\operatorname{Tr}^{(q)} A(x,x)dv_X(x)\] and put 
\begin{eqnarray}
\operatorname{Tr}\lbrack A\rbrack:= \sum_{j=0}^n \operatorname{Tr}^{(j)}[A], \nonumber \\
\operatorname{STr}\lbrack A\rbrack:= \sum_{j=0}^n (-1)^j \operatorname{Tr}^{(j)}[A].
\end{eqnarray}
Let 
$$
\Pi_m : L^2_m(X, T^{*0,\bullet}X \otimes E) \to \Ker \Box_{b,m}
$$  
be the orthogonal projection and let 
$$
\Pi^\perp_m : L^2_m(X, T^{*0,\bullet}X \otimes E) \to (\Ker\Box_{b,m})^\perp
$$  
be the orthogonal projection, where 
\[(\Ker\Box_{b,m})^\perp=\set{u\in L^2_m(X, T^{*0,\bullet}X \otimes E);\, (\,u\,|\,v\,)_E=0,\  \ \forall v\in\Ker \Box_{b,m}}.\] 

By \cite[Theorem 1.7]{CHT}, we have the following asymptotic expansion: 
\begin{equation}\label{E:5.5.10}
\operatorname{STr}  \lbrack N e^{-t \Box_{b,m}}  \rbrack\sim\sum^\infty_{j=0} \hat{B}_{m,-n+\frac{j}{2}}t^{-n+\frac{j}{2}}\ \ \mbox{as $t\To0^+$}, 
\end{equation}
where $\hat{B}_{m,-n+\frac{j}{2}}\in\Complex$ is independent of $t$, $j=0,1,2,\ldots$. 

\begin{lemma}\label{l-gue160420}
Fix $q=0,1,\ldots,n$. For every $\delta>0$, there exist $c>0$, $C>0$ such that 
\begin{equation}\label{e-gue160420Ia}
\abs{\operatorname{Tr^{(q)}}  \lbrack e^{-t \Box_{b,m}} \Pi^\perp_m \rbrack}\leq Ce^{-ct},\ \ \forall t\geq\delta. 
\end{equation}
\end{lemma}

\begin{proof}
Let $0<\lambda_1\leq\lambda_2\leq\lambda_3\leq\cdots$ be the non-zero eigenvalues of $\Box^{(q)}_{b,m}$. For $t>0$, we have 
\begin{equation}\label{e-gue160421sI}
\abs{\operatorname{Tr^{(q)}}  \lbrack e^{-t \Box_{b,m}} \Pi^\perp_m \rbrack}=e^{-\lambda_1t}+e^{-\lambda_2t}+\cdots.
\end{equation}
Let $\delta>0$. From \eqref{e-gue160421sI}, for $t\geq\delta$, we have 
\begin{equation}\label{e-gue160421sII}
\begin{split}
\abs{\operatorname{Tr^{(q)}}  \lbrack e^{-t \Box_{b,m}} \Pi^\perp_m \rbrack}&=e^{-\lambda_1t}+e^{-\lambda_2t}+\cdots\\
&=e^{-\lambda_1\frac{t}{2}}(e^{-\lambda_1\frac{t}{2}}+e^{-\lambda_2t+\lambda_1\frac{t}{2}}+e^{-\lambda_3t+\lambda_1\frac{t}{2}}+\cdots)\\
&\leq e^{-\lambda_1\frac{t}{2}}(e^{-\lambda_1\frac{t}{2}}+e^{-\lambda_2\frac{t}{2}}+e^{-\lambda_3\frac{t}{2}}+\cdots)\\
&\leq e^{-\lambda_1\frac{t}{2}}(e^{-\lambda_1\frac{\delta}{2}}+e^{-\lambda_2\frac{\delta}{2}}+e^{-\lambda_3\frac{\delta}{2}}+\cdots)\\
&=e^{-\lambda_1\frac{t}{2}}\abs{\operatorname{Tr^{(q)}}  \lbrack e^{-\frac{\delta}{2}\Box_{b,m}} \Pi^\perp_m \rbrack}.
\end{split}
\end{equation}
From \eqref{e-gue160421sII}, the lemma follows. 
\end{proof}

From \eqref{E:5.5.10} and Lemma~\ref{l-gue160420}, we see that $\operatorname{STr}  \lbrack N e^{-t \Box_{b,m}} \Pi^\perp_m \rbrack$ satisfies \eqref{e-gue160420g} and \eqref{e-gue160420I}. By Definition \ref{d-gue160313}, for $\operatorname{Re}(z)>n$, we can define 
\begin{equation}\label{E:5.5.12}
\theta_{b,m}(z)  = - M \left\lbrack \operatorname{STr}  \lbrack N e^{-t \Box_{b,m}} \Pi^\perp_m  \rbrack \right\rbrack  =   - \operatorname{STr} \left\lbrack N ({\Box}_{b,m})^{-z} {\Pi}^\perp_m \right\rbrack.  
\end{equation}
By Theorem~\ref{t-gue160313}, we have the following lemma.

\begin{lemma}\label{L:mero}
$\theta_{b,m}(z)$ extends to a meromorphic function on $\mathbb{C}$ with poles contained in the set  
\[\set{\ell-\frac{j}{2};\, \ell,j\in\mathbb Z},\]
its possible poles are simple, and $\theta_{b,m}(z)$ is holomorphic at $0$. Moreover,
\begin{equation}\label{E:5.5.13}
\begin{split}
& \theta_{b,m}'(0)  =  -\int_0^1 \left\{  \operatorname{STr} \Big[  Ne^{-t\Box_{b,m}}\Big]-\sum^{2n}_{j=0} \hat{B}_{m,-n+\frac{j}{2}}t^{-n+\frac{j}{2}}\right\} \frac{dt}{t}  \\
& -\int_1^\infty \operatorname{STr} \Big[ Ne^{-t\Box_{b,m}} \Pi_m^\perp \Big] \frac{dt}{t} - \sum^{2n-1}_{j=0} \hat{B}_{m,-n+\frac{j}{2}}\frac{1}{\frac{j}{2}-n} + \Gamma'(1)(\hat{B}_{m,0} - \operatorname{STr}[N\Pi_m]). 
\end{split}
\end{equation}
\end{lemma}

\begin{definition}\label{d-gue160502w}
Fix $m\in\mathbb Z$. We define $\exp ( -\frac{1}{2} \theta_{b,m}'(0) )$ the $m$-th Fourier component of the analytic torsion for the rigid vector bundle $E$ over the CR manifold $X$ with transversal CR $S^1$-action. 
\end{definition}

\section{The asymptotics of the analytic torsion}

Recall that we work with the assumption that $X$ is strongly pseudoconvex. We can repeat the proof of Theorem~\ref{t-gue160423} and deduce 

\begin{theorem}\label{t-gue160427}
With the notations used before, we have
\begin{equation}\label{e-gue160427}
m^{-n}{\rm STr\,}\lbrack Ne^{-\frac{t}{m}\Box_{b,m}}\rbrack\sim\sum^\infty_{j=0}B_{m,-n+\frac{j}{2}}t^{-n+\frac{j}{2}}\ \ \mbox{as $t\To0^+$, uniformly in $m$},
\end{equation}
where for each $j=0,\frac{1}{2},1,\ldots$, $B_{m,-n+j}\in\Complex$ is independent of $t$ and there is a constant $\hat C_j>0$ independent of $m$, such that $\abs{B_{m,-n+j}}\leq\hat C_j$, for every $m\in\mathbb N$, and 
\begin{equation}\label{e-gue160427I}
\begin{split}
&\lim_{m\To\infty}B_{m,-n+\frac{j}{2}}=\operatorname{rk}(E)\int_XNA_{-n+\frac{j}{2}}(x)dv_X,\ \ \mbox{if $j$ is an even number},\\
&\lim_{m\To\infty}B_{m,-n+\frac{j}{2}}=0,\ \ \mbox{if $j$ is an odd number},
\end{split}
\end{equation}
where $A_\ell(x)\in C^\infty(X,{\rm End\,}(T^{*0,\bullet}X)\otimes E)$ is as in \eqref{E:5.5.34}, $\ell=-n,-n+1,\ldots$.
\end{theorem}

\subsection{Asymptotics of the analytic torsion}

For $x \in X, t > 0$, we set
\begin{equation}\label{E:5.5.37}
\mathcal{R}_t(x) = \det \left( \frac{\dot{\mathcal{R}}}{2 \pi} \right) \operatorname{Tr} \Big[(\operatorname{Id}-\exp (t \dot{\mathcal{R}}))^{-1}\Big].
\end{equation}

Let $\omega^1(x),\ldots,\omega^n(x)\in C^\infty(X,T^{*1,0}X)$ be an orthonormal basis for $T^{*1,0}_xX$, for every $x\in X$. Define $\Theta(x):=i\sum^n_{j=1}\omega^j(x)\wedge\ol{\omega^j}(x)\in C^\infty(X,T^{*1,1}X)$. Then we have
\begin{equation}\label{E:5.5.38}
dv_X =\frac{1}{n!}\Theta^n \wedge(-\omega_0), \quad \det \left( \frac{\dot{\mathcal{R}}}{2\pi} \right)dv_X(x) = \frac{1}{n!} \left(-\frac{1}{2\pi} d\omega_0 \right)^n \wedge(-\omega_0).
\end{equation}
By \eqref{E:5.5.37} and \eqref{E:5.5.38}, for any $k \in \mathbb{N}$, we have the following asymptotic expansion, for $t \to 0$, 
\begin{equation}\label{E:5.5.39}
\mathcal{R}_t(x) = \sum_{\ell=-n}^k \widehat{A}_\ell(x) t^l + O(t^{k+1}),
\end{equation} 
where 
\begin{equation}\label{E:5.5.40}
\begin{split}
& \widehat{A}_\ell = 0  \ \   \text{for} \  \   l \le -2, \quad \widehat{A}_{-1} dv_X = -\frac{1}{(n-1)!} \frac{\Theta}{2\pi} \wedge \left(-\frac{1}{2\pi}d\omega_0 \right)^{n-1}\wedge(-\omega_0), \\
& \widehat{A}_0 \ dv_X = \frac{n}{2}\frac{1}{ n!} \left(-\frac{1}{2\pi}d\omega_0 \right)^n \wedge(-\omega_0).
\end{split} 
\end{equation}

We need

\begin{lemma}\label{l-gue160428}
With the notations above, we have
\begin{equation}\label{E:5.5.45}
(2\pi)^{-n} \frac{\det(\dot{\mathcal{R}})\operatorname{STr} Ne^{t\gamma_d}}{\det(1-\exp(-t\dot{\mathcal{R}}))}(x) = \mathcal{R}_t(x),\ \ \forall x\in X.
\end{equation}
\end{lemma}

\begin{proof}
Fix $p\in X$, let $\{ \omega_j(x) \}_{j=1}^n$ to be an orthonormal basis of $T^{1,0}_xX$ such that
\begin{equation}\label{e-gue160428}
\dot{\mathcal{R}}(p) = \operatorname{diag} (\mu_1, \cdots, \mu_n) \in \operatorname{End}(T_p^{1,0}X),
\end{equation}\
then
\begin{equation}\label{e-gue160428I}
\begin{split}
&\gamma_d(p) = -\sum^n_{j=1} \mu_j \overline{\omega}^j(p) \wedge \iota_{\overline{\omega}_j}(p),\\
&{\rm det\,}\dot{\mathcal{R}}(p):=\mu_1\cdots\mu_n.
\end{split}
\end{equation}
From \eqref{e-gue160428} and \eqref{e-gue160428I}, it is easy to check that 
\begin{equation}\label{e-gue160428II}
\mathcal{R}_t(p) =\frac{\mu_1\cdots\mu_n}{(2\pi)^n}\Bigr(\frac{1}{1-e^{\mu_1 t}}+\cdots+\frac{1}{1-e^{\mu_n t}}\Bigr)
\end{equation}
and 
\begin{equation}\label{e-gue160428III}
\begin{split}
&(2\pi)^{-n} \frac{\det(\dot{\mathcal{R}})\operatorname{STr} Ne^{t\gamma_d}}{\det(1-\exp(-t\dot{\mathcal{R}}))}(p)\\
&=\frac{\mu_1\cdots\mu_n}{(2\pi)^n}\frac{1}{(1-e^{-\mu_1 t})\cdots(1-e^{-\mu_n t})}\sum^n_{q=0}(-1)^qq\sum_{J=(j_1,\ldots,j_q),1\leq j_1<j_2<\cdots<j_q\leq n}e^{(-\mu_{j_1}-\cdots-\mu_{j_q})t}\\
&=\frac{\mu_1\cdots\mu_n}{(2\pi)^n}\frac{e^{(\mu_1+\cdots+\mu_n)t}}{(1-e^{\mu_1 t})\cdots(1-e^{\mu_n t})}\sum^n_{q=0}(-1)^{n+q}q\sum_{J=(j_1,\ldots,j_q),1\leq j_1<j_2<\cdots<j_q\leq n}e^{(-\mu_{j_1}-\cdots-\mu_{j_q})t}.
\end{split}
\end{equation}
Now,
\begin{equation}\label{e-gue160428IV}
\begin{split}
&e^{-(\mu_1+\cdots+\mu_n)t}\sum^n_{j=1}(1-e^{\mu_1 t})\cdots(1-e^{\mu_{j-1} t})(1-e^{\mu_{j+1} t})\cdots(1-e^{\mu_n t})\\
&=e^{-(\mu_1+\cdots+\mu_n)t}\sum^n_{j=1}\sum^n_{q=0}(-1)^q\sum_{J=(j_1,\ldots,j_q),1\leq j_1<\cdots<j_q\leq n, j\notin J}e^{(\mu_{j_1}+\cdots+\mu_{j_q})t}\\
&=\sum^n_{j=1}\sum^n_{q=0}(-1)^{n+q}\sum_{J=(j_1,\ldots,j_q),1\leq j_1<\cdots<j_q\leq n, j\in J}e^{(-\mu_{j_1}-\cdots-\mu_{j_q})t}\\
&=\sum^n_{q=0}(-1)^{n+q}q\sum_{J=(j_1,\ldots,j_q),1\leq j_1<j_2<\cdots<j_q\leq n}e^{(-\mu_{j_1}-\cdots-\mu_{j_q})t}.
\end{split}
\end{equation}
From \eqref{e-gue160428IV}, \eqref{e-gue160428III} and \eqref{e-gue160428II}, the lemma follows. 
\end{proof}

\begin{theorem}\label{T:5.5.10}
With the notations used in Theorem~\ref{t-gue160427}, we have 
\begin{equation}\label{E:5.5.43}
\lim_{m\To\infty}B_{m,-n+\frac{j}{2}}=\frac{1}{2\pi}\operatorname{rk}(E)\int_X\widehat{A}_{-n+\frac{j}{2}}(x),\quad \text{for all $j\geq 0$, $j$ is even}
\end{equation}
where $\widehat{A}_{\ell}$ is as in \eqref{E:5.5.39}. 
\end{theorem}
\begin{proof}
By \eqref{E:5.5.34} and Lemma~\ref{l-gue160428}, we get
\begin{equation}\label{E:5.5.46}
\operatorname{STr} NA_\ell(x) = \frac{1}{2\pi}\widehat{A}_\ell(x),\ \ \ell=-n,-n+1,\ldots.
\end{equation}
From \eqref{E:5.5.46} and \eqref{e-gue160427I}, the theorem follows. 
\end{proof}

\begin{theorem}\label{T:5.5.11}
There exist $C, c, c'>0$ such that for any $q \ge 1, t \ge 1, m \in \mathbb{N}$, we have
\begin{equation}\label{E:5.5.48}
m^{-n} \operatorname{Tr}^{(q)} \Big[ e^{-\frac{t}{m}\Box_{b,m}}  \Big] \le C \exp\left( -(c-c'/m)t \right).
\end{equation}
\end{theorem}
\begin{proof}
By Theorem \ref{T:1.5.5}, for $t \ge 1, q \ge 1$, we have 
\begin{equation}\label{E:5.5.49}
\operatorname{Tr}^{(q)} \Big[ e^{-\frac{t}{m}\Box_{b,m}} \Big] \le \operatorname{Tr}^{(q)} \Big[e^{-\frac{1}{m}\Box_{b,m}} \Big] \exp \left( -\frac{(t-1)}{m}(c_1m-c_2) \right).
\end{equation} 
By \eqref{E:1.6.4}, we know that $m^{-n}\operatorname{Tr}^{(q)} \Big[e^{-\frac{1}{m}\Box_{b,m}}\Big]$ has a finite limit as $m \to+\infty$. By \eqref{E:5.5.49}, we have \eqref{E:5.5.48}.
\end{proof}

The main result of this work is the following

\begin{theorem}\label{T:5.5.8}
As $m \to+\infty$, we have
\begin{equation}\label{E:5.5.33}
\theta_{b,m}'(0) =  \frac{\operatorname{rk}(E)}{4\pi}
 \int_X \log \left( \det \left(  \frac{m \dot{\mathcal{R}} }{2\pi} \right) \right) e^{-m\frac{d\omega_0}{2\pi}} \wedge (-\omega_0) + o(m^{n}),
\end{equation}
where $\dot{\mathcal{R}} \in \operatorname{End}(T^{1,0}X)$ is defined in \eqref{E:1.5.15}.
\end{theorem}

\begin{proof}
For $m > \frac{2c_2}{c_1}$ in Theorem \ref{T:1.5.5}, set
\begin{equation}\label{E:5.5.50}
\widetilde{\theta}_{b,m}(z) = -M \left\lbrack m^{-n} \operatorname{STr} \Big\lbrack Ne^{-\frac{t}{m}\Box_{b,m}}\Pi_m^\perp \Big\rbrack \right\rbrack(z).
\end{equation}
Clearly 
\begin{equation}\label{E:5.5.51}
m^{-n} \theta_{b,m}(z) = m^{-z} \widetilde{\theta}_{b,m}(z).
\end{equation}
By Theorem \ref{T:1.5.5}, Theorem \ref{t-gue160313}, Theorem \ref{t-gue160427}, Theorem~\ref{T:5.5.10}, \eqref{E:5.5.40} and \eqref{E:5.5.51}, for $m > \frac{2c_2}{c_1}$, we have
\begin{equation}\label{E:5.5.52}
\begin{split}
& m^{-n} \theta'_{b,m}(0)  = - \log(m) \widetilde{\theta}_{b,m}(0) + \widetilde{\theta}'_{b,m}(0),   \\
& \lim_{m \to \infty} \widetilde{\theta}_{b,m}(0) =-\frac{1}{2\pi}\operatorname{rk}(E)\int_X\widehat{A}_0dv_X=-\frac{1}{2\pi}\operatorname{rk}(E)\frac{n}{2}\frac{1}{n!}\int_X(-\frac{1}{2\pi}d\omega_0)^n\wedge(-\omega_0),\\
&\operatorname{STr} \lbrack N\Pi_m \rbrack=0.
\end{split}
\end{equation} 
By Theorem~\ref{t-gue160427}, Theorem~\ref{T:5.5.10}, \eqref{E:5.5.13} and Lebesgue's dominated convergence theorem, we get
\begin{equation}\label{E:5.5.53}
\begin{split}
 \lim_{m \to \infty} \widetilde{\theta}'_{b,m}(0) = & - \lim_{m \to \infty} \int_0^1 \left\{  m^{-n} \operatorname{STr} \left\lbrack Ne^{-\frac{t}{m}\Box_{b,m}}\right\rbrack - \sum^{2n}_{j=0} B_{m,-n+\frac{j}{2}}t^{-n+\frac{j}{2}}\right\} \frac{dt}{t} \\
& -\lim_{m \to \infty}  \int_1^\infty \left\{  m^{-n} \operatorname{STr} \left\lbrack Ne^{-\frac{t}{m}\Box_{b,m}} \right\rbrack \right\} \frac{dt}{t}  \\
&-\lim_{m\to\infty}\sum^{2n-1}_{j=0}B_{m,-n+\frac{j}{2}}\frac{1}{\frac{j}{2}-n}+\Gamma'(1)\lim_{m\to\infty}B_{m,0}\\
&= - \int_0^1 \lim_{m \to \infty}\left\{  m^{-n} \operatorname{STr} \left\lbrack Ne^{-\frac{t}{m}\Box_{b,m}} \right\rbrack - \sum^{2n}_{j=0} B_{m,-n+\frac{j}{2}}t^{-n+\frac{j}{2}}\right\} \frac{dt}{t} \\
& -  \int_1^\infty\lim_{m \to \infty}\left\{  m^{-n} \operatorname{STr} \left\lbrack Ne^{-\frac{t}{m}\Box_{b,m}} \right\rbrack \right\} \frac{dt}{t}  \\
&+\frac{1}{2\pi}\operatorname{rk}(E)\int_X\widehat{A}_{-1}dv_X+\frac{1}{2\pi}\operatorname{rk}(E)\Gamma'(1)\int_X\widehat{A}_0dv_X.
\end{split}
\end{equation}
By Theorem \ref{T:1.6.1} and \eqref{E:5.5.45}, for any $t>0$,
\begin{equation}\label{E:5.5.54}
\lim_{m \to \infty} m^{-n} \operatorname{STr} \left\lbrack Ne^{-\frac{t}{m}\Box_{b,m}}(x,x)  \right\rbrack  \  =\frac{1}{2\pi}\operatorname{rk}(E) \mathcal{R}_t(x),
\end{equation}
and the convergence is uniform for $x \in X$ and for $t$ varying in compact subsets of $]0, \infty[.$ 

For $z \in \mathbb{C}$ set
\begin{equation}\label{E:5.5.55}
\widetilde{\zeta}(z)  \  =  \  -M \left\lbrack \int_X \mathcal{R}_t(x)dv_X(x) \right\rbrack (z).
\end{equation}

By Theorem \ref{T:5.5.10}, Theorem \ref{T:5.5.11}, \eqref{E:5.5.53}, \eqref{E:5.5.54}, \eqref{E:5.5.55} and \eqref{e-gue160421s}, we have
\begin{equation}\label{E:5.5.56}
\begin{split}
 &\lim_{m\to\infty}\widetilde{\theta}'_{b,m}(0)\\
 =& - \frac{1}{2\pi}\operatorname{rk}(E) \int_0^1 \left\{ \int_X \mathcal{R}_t(x) dv_X(x)-\left(\frac{1}{t}\int_X\widehat{A}_{-1}dv_X + \int_X\widehat{A}_0dv_X \right) \right\}\frac{dt}{t} \\
& -\frac{1}{2\pi}\operatorname{rk}(E) \int_1^\infty \int_X \mathcal{R}_t(x) dv_X(x) \frac{dt}{t} + \frac{1}{2\pi}\operatorname{rk}(E)\int_X\widehat{A}_{-1}dv_X+\frac{1}{2\pi}\operatorname{rk}(E)\Gamma'(1)\int_X\widehat{A}_0dv_X\\
 = &\frac{1}{2\pi}\operatorname{rk}(E) \widetilde{\zeta}'(0).
\end{split} 
\end{equation}

Since $\dot{\mathcal{R}} \in \operatorname{End}(T^{1,0}X)$ has positive eigenvalues, we can check that, for $\operatorname{Re}z>1,$
\begin{equation}\label{E:5.5.57}
\widetilde{\zeta}(z) = \left( \int_X \det \left( \frac{\dot{\mathcal{R}}}{2\pi} \right)\operatorname{Tr}(\dot{\mathcal{R}} )^{-z}(x)dv_X(x) \right) \frac{1}{\Gamma(z)}\int_0^\infty t^{z-1}\frac{e^{-t}}{1-e^{-t}}dt.
\end{equation}
Let $\zeta(z) =\sum_{k=1}^\infty \frac{1}{n^z}$ be the Riemann zeta function. It is well known that, for $\operatorname{Re}z>1$,
\begin{equation}\label{E:5.5.58}
\zeta(z) =\frac{1}{\Gamma(z)} \int_0^\infty t^{z-1} \frac{e^{-t}}{1-e^{-t}}dt.
\end{equation} 
Moreover,
\begin{equation}\label{E:5.5.59}
\zeta(0) = -\frac{1}{2}, \qquad \zeta'(0) = -\frac{1}{2} \log (2\pi).
\end{equation}

By \eqref{E:5.5.57}, \eqref{E:5.5.58} and \eqref{E:5.5.59}, we get
\begin{equation}\label{E:5.5.60}
\begin{split}
\widetilde{\zeta}'(0) = &  - \zeta(0)  \int_X \det \left( \frac{\dot{\mathcal{R}}}{2\pi}\right) \operatorname{Tr} \left\lbrack \log \dot{\mathcal{R}} \right\rbrack dv_X(x) + n \zeta'(0)  \int_X \det \left( \frac{\dot{\mathcal{R}}}{2\pi} \right) dv_X(x)  \\
 = & \frac{1}{2}\int_X \det \left( \frac{\dot{\mathcal{R}}}{2\pi}\right) \log \left( \det \left( \frac{\dot{\mathcal{R}}}{2\pi}\right)  \right) dv_X(x).
\end{split}
\end{equation}
By \eqref{E:5.5.52}, \eqref{E:5.5.56}, \eqref{E:5.5.60} and \eqref{E:5.5.38}, we get \eqref{E:5.5.33}.
\end{proof}

Let's come back to complex geometry case. Let $M$ be a compact complex manifold of dimension $n$ and let $G$ be a holomorphic vector bundle over $M$ with  a Hermitian metric $h^G$.
Let $\Box_G$ be the Kodaira Laplacian with values in $T^{*0,\bullet}M\otimes G$ and let $e^{-t\Box_G}$ be the associated heat operator. It is well-known that $e^{-t\Box_G}$ admits an asymptotic expansion as $t\To0^+$. The analytic torsion associated to $G$ is defined by $\exp ( -\frac{1}{2} \theta_{G}'(0) )$, where 
\[\theta_{G}(z)  = - M \left\lbrack \operatorname{STr}  \lbrack N e^{-t \Box_G} P^\perp  \rbrack \right\rbrack.\] 
Here $N$ is the number operator on $T^{*0,\bullet}M$, $\operatorname{STr}$ denotes the super trace , $P^\perp$ is the orthogonal projection onto $({\rm Ker\,}\Box_G)^\perp$ and $M$ denotes the Mellin transformation. 

From Theorem \ref{T:5.5.8}, we deduce Bismut and Vasserot's asymptotic formula (see \cite[Theorem 8]{BV} and \cite[Theorem 5.5.8]{MM}):
\begin{corollary} 
With the notations above, let $M$ be a compact Hermitian manifold of dimension $\dim_\mathbb{C} M = n$. Let $(L, h^L)$ be a Hermitian line bundle over $M$ such that the curvature $R^L$ associated to $h^L$ is positive, i.e. $\sqrt{-1}R^L$ is a positive $(1, 1)$-form and let $(F, h^F)$ be a holomorphic Hermitian vector bundle on $M$. 
As $m \to+\infty$, we have 
\[
\theta'_{L^m\otimes F}(0)  =  \frac{\operatorname{rk}F}{2}
 \int_M \log \left( \det \left(  \frac{m \dot{R}^L }{2\pi} \right) \right) e^{m\frac{\sqrt{-1}}{2\pi}R^L}  + o(m^{n}),
\]
where $\dot{R}^L$ is defined in \cite[(1.5.15)]{MM}.
\end{corollary}

\begin{proof}
We take $X$ to be the circle bundle $\left\{ v \in L^* \ : \ |v|^2_{h^{-1} } =1 \right\}$ over the compact complex manifold $M$ where $(L^*, h^{-1})$ is the dual line bundle of Hermitian line bundle $(L, h)$ over $M$. Let $(z, \lambda)$ be the local coordinates on $L^*$, where $\lambda$ is the fiber coordinates. The natural $S^1$-action on $X$ is defined by $e^{i\theta} \circ (z, \lambda) \ = \ (z, e^{i\theta}\lambda)$. Then we can check that $X$ is a compact CR manifold with a transversal CR $S^1$-action. Let $x = (z, \lambda) \in X$. On $X$ we can check that $\mathcal{L}_x|_{T^{1,0}X} = \frac{1}{2} R^L_z$. Let $E$ be the bundle over $X$ such that $E_x = F_z \times \left\{ \lambda \right\}$. It is well known that (see \cite[p. 746]{MM0}, Theorem 1.2 in \cite{CHT}), the spectrum of the Dolbeault Laplacian $\Box_{L^m\otimes F}$ associated to the bundle $L^{\otimes m} \otimes F$ on $M$ is the same as the spectrum of the $\Box_{b,m}$ associated to the bundle $E$ on $X$. From \eqref{E:5.5.33}, we have
\begin{eqnarray}
\theta'_{L^m\times F}(0) \ =  \ \theta_{b,m}'(0) & = & \frac{\operatorname{rk} E}{4\pi}
 \int_X \log \left( \det \left(  \frac{m \dot{\mathcal{R}} }{2\pi} \right) \right) e^{-m\frac{d\omega_0}{2\pi}} \wedge (-\omega_0) + o(m^{n}) \nonumber \\
& = & \frac{\operatorname{rk}F}{2} \int_M \log \left( \det \left(  \frac{m \dot{R}^L }{2\pi} \right) \right) e^{m\frac{\sqrt{-1}}{2\pi}R^L} + o(m^{n}) \nonumber
\end{eqnarray}

\end{proof}


\bibliographystyle{plain}

\end{document}